\documentclass[a4paper,twoside,12pt]{amsart}

\usepackage{amssymb}
\usepackage{amsmath}
\usepackage{graphicx}
\usepackage{hyperref}
\usepackage{subfigure}

\usepackage[latin1]{inputenc} 
\usepackage[T1]{fontenc}

\newtheorem{theorem}{Theorem}[section]
\newtheorem{lemma}[theorem]{Lemma}
\newtheorem{prop}[theorem]{Proposition}
\newtheorem{cor}[theorem]{Corollary}

\theoremstyle{definition}
\newtheorem{definition}[theorem]{Definition}

\newcommand\bn{{\mathbb N}}
\newcommand\bz{{\mathbb Z}}
\newcommand\bq{{\mathbb Q}}
\newcommand\bc{{\mathbb C}}

\newcommand\ba{{\mathbb A}}

\newcommand\br{{\mathbb R}}

\newcommand\Ic{{\mathcal I}}

\newcommand\nc{{\mathcal N}}

\newcommand\cl{{\mathcal L}}

\newcommand\cls{{\mathcal S}}
\newcommand\Gc{{\mathcal G}}

\newcommand\pc{{\mathcal P}}

\newcommand\bl{{\mathbb L}}

\newcommand\tS{{\operatorname{Supp}}}

\newcommand\ti{{\operatorname{in}}}

\newcommand\ag{{\alpha}}
\newcommand\bg{{\beta}}
\newcommand\Dg{{\Delta}}
\newcommand\mg{{\mu}}
\newcommand\sg{{\sigma}}
\newcommand\Sg{{\Sigma}}

\newcommand\og{{\omega}}
\newcommand\zg{{\zeta}}

\numberwithin{equation}{section}

\DeclareMathOperator\spec{Spec}

\theoremstyle{remark}
\newtheorem{remark}[theorem]{Remark}

\numberwithin{equation}{section}

\title[The Newton tree: geometric interpretation and applications]{The Newton tree: geometric interpretation and applications to the motivic zeta function and the log canonical threshold}

\author{Pierrette  Cassou-Nogu\`es}
\address{Institut de Math\'ematiques de Bordeaux, Universit\'e Bordeaux I,
350, Cours de la Lib\'eration, 33405, Talence Cedex 05, FRANCE}
\thanks{}

\author{Willem Veys}
\address{KU Leuven, Dept. Wiskunde, Celestijnenlaan 200B, 3001 Leuven, Belgium}
\curraddr{}
\thanks{}

\thanks{
First author is partially supported  by the grants
MTM2010-21740-C02-01 and MTM2010-21740 -C02-02. Second author is
partially supported by the KU Leuven grant OT/11/069.}


\dedicatory{}

\subjclass[2000]{}

\begin{document}
\begin{abstract}
Let $\Ic$ be an arbitrary ideal in $\bc[[x,y]]$. We use the Newton
algorithm to compute by induction the motivic zeta function of the
ideal, yielding only few poles, associated to the faces of the
successive Newton polygons. We associate a minimal Newton tree to
$\Ic$, related to using good coordinates in the Newton algorithm,
and show that it has a conceptual geometric interpretation in terms
of the log canonical model of $\Ic$. We also compute the log
canonical threshold from a Newton polygon and strengthen Corti's
inequalities.
\end{abstract}

\maketitle

\section{Introduction}

We introduced in \cite{CV} the Newton algorithm as an efficient way
to study an arbitrary ideal $\Ic$ in $\bc[[x,y]]$, involving a
finite succession of Newton polygons. We showed for instance how to
compute its integral closure and the Hilbert-Samuel multiplicity of
$\Ic$ in a nice combinatorial way. In this article we provide more
applications. First we show how the Newton algorithm can be used to
compute by induction the motivic zeta function of $\Ic$, yielding
candidate poles associated to the faces of the successive Newton
polygons (Theorem \ref{formulazetafunction}). This method is much
more efficient than other ones in the literature, and almost all
candidate poles are true poles. The classical method uses a log
principalization of $\Ic$ and yields generally an enormous amount of
false candidate poles. On the other hand the formula in
\cite[Theorem 6.1]{V1} has only the actual poles, but it involves
horrible determinants, that are in fact again expressed in terms of
a log principalization.

 \smallskip
In \cite{CV} we moreover codified most of the data of the algorithm in a useful combinatorial object, the Newton tree of $\Ic$. This (decorated) tree depends heavily on the chosen resolution. We introduce an operator on Newton trees, {\em exchange of vertical edge}, and show that it corresponds to a change of coordinates (Proposition \ref{exchange}). We define the notion of {\em minimal Newton tree}, we prove that any Newton tree can be converted into a minimal one by exchanging vertical edges (Proposition \ref{towardsminimal}), and that this is related to the use of so-called {\em (very) good coordinates}. Coming back to the motivic zeta function, some very particular faces do not give poles (Proposition \ref{no-contribution}), and these faces will not appear when one uses very good coordinates.

A minimal Newton tree of $\Ic$ has also a conceptual geometric interpretation. Roughly its vertices correspond to the exceptional components in the log canonical model of the ideal, and its edges to intersections of components (Theorem \ref{principalization}). Also the decorations of the Newton tree correspond exactly to the \lq classical\rq\ decorations on a dual tree of a curve configuration. For vertex decorations this was shown in \cite{CNL} in the context of principal ideals; the argument is still valid for arbitrary ideals. For edge decorations the result is new (Proposition \ref{edgedecorations}).

 \smallskip

We also show a generalization of a result known in the case where the ideal is principal \cite{K}\cite{ACLM3}, that is, that there exists a system of coordinates such that the log canonical threshold can be computed from the Newton polygon (Theorem \ref{lct}). In \cite{dFEM} and \cite{Co} the authors prove inequalities between the Hilbert-Samuel multiplicity of an ideal and its log canonical threshold. Using the computations via Newton polygons these inequalities appear to be simple geometric facts (Corollary \ref{HSknown}), and using the computations via the Newton algorithm they can be strengthened (Theorem \ref{HSinequality}).

\medskip
The article is organized as follows. In the next section we recall
the Newton algorithm for ideals. In section $3$ we apply the Newton
algorithm to the computation of the motivic zeta function. In
section 4, after recalling the construction of Newton trees, we
compare them in different systems of coordinates and we introduce
the notion of minimal Newton tree. In section $5$ we show that the
dual tree of the log canonical model of the ideal can be deduced
from a minimal Newton tree. Finally, we study the log canonical
threshold and the inequalities involving the Hilbert-Samuel
multiplicity in section 6.

\medskip
\section{Newton Algorithm for an ideal}

\medskip
We recall briefly the essential definitions and results from
\cite{CV} concerning the Newton algorithm.

\smallskip
\subsection{Newton polygon}
For any set $E\subset \bn \times \bn$, denote by $\Delta(E)$ the
smallest convex set containing $E+\br _{+}^2=\{a+b \mid a\in E,\
b\in \br _{+}^2\}$. A set $\Delta \subset \br ^2$ is a { \it Newton
diagram} if there exists a set  $E\subset \bn \times \bn$ such that
$\Delta=\Delta (E)$. The smallest set $E_0\subset \bn \times \bn$
such that $\Delta=\Delta (E_0)$ is called the set of {\em vertices}
of a Newton diagram $\Delta$; it is a finite set. Let
$E_0=\{v_0,\cdots, v_m\}$, with  $v_i=(\ag _i,\bg_i)\in \bn \times
\bn$ for $i=0,\cdots, m$, and $\ag_{i-1}<\ag_i$, $\bg_{i-1}>\bg_i$
for $i=1,\cdots, m$.

For $i\in\{1,\cdots, m\}$, denote $S_i=[v_{i-1},v_i]$ and by
$l_{S_i}$ the line supporting the segment $S_i$. We call $\nc
(\Dg)=\cup_{1\leq i\leq m} S_i$ the {\it Newton polygon} of $\Dg$
and the $S_i$ its {\it faces}. The Newton polygon $\nc (\Dg)$ is
empty if and only if $E_0 =\{v_0=(\ag_0,\bg_0)\}$. The Newton
diagram $\Delta$ has also two non-compact faces: the vertical
half-line starting at $v_0$, and the horizontal half-line starting
at $v_m$. The integer $h(\Dg)=\bg _0-\bg_m$ is called the {\it
height} of $\Dg$.
 Let
 $$f(x,y)=\sum_{(\ag,\bg)\in \bn \times \bn}c_{\ag,\bg}x^{\ag}y^{\bg}\in \bc[[x,y]] .$$
 We define the support of $f$ as
 $$\tS f=\{(\ag,\bg)\in \bn \times \bn \mid c_{\ag,\bg}\neq 0\} .$$
 We denote $\Delta(f)=\Delta(\tS f)$ and $\nc (f)=\nc(\Dg (f))$.
 Let $l$ be a line in $\br^2$. We define the initial part of $f$ with respect to $l$ as
 $$\ti (f,l)=\sum_{(\ag,\bg)\in l}c_{\ag,\bg}x^{\ag}y^{\bg} .$$
 If the line $l$ has equation $p\ag+q\bg =N$, with $(p,q)\in (\bn^*)^2$ and $\gcd (p,q)=1$, then $\ti (f,l)$ is zero or a monomial or, if $l=l_S$ for some segment $S$ of $\nc (\Dg)$, of the form
 $$\ti (f,l)=x^{a_l}y^{b_l}F_S(x^q,y^p) ,$$
   where $(a_l,b_l)\in \bn^2$ and
 $$F_S(x,y)=c\prod _{1\leq i\leq n }(y-\mg _ix)^{\nu_i} ,$$
 with $c\in \bc^*$, $n\in \bn^*$, $\mg_i\in \bc^*$ (all different) and $\nu_i \in \bn^*$.

 \bigskip
 Now let  $\Ic=(f_1,\cdots, f_r)$ be a non-trivial ideal in $\bc[[x,y]]$. We define
 $$\Dg(\Ic)=\Dg(\cup _{1\leq i\leq r}\tS f_i)  \quad\text{and}\quad \nc(\Ic)=\nc(\Dg(\Ic)) .$$
 When $\Ic=(f)$ we simply write $\Dg(f)$ and $\nc(f)$.
For a segment $S$ of $\nc(\Ic)$ we denote by $\ti (\Ic,S)$ the ideal generated by the
 $\ti (f_i,l_S), 1\leq i\leq r$, and call it the {\it initial ideal} of $\Ic$ with respect to $S$.

One easily verifies that $\Dg(\Ic)=\Dg(\cup _{f\in \Ic}\tS f)$;
hence the sets $\Dg(\Ic)$ and $\nc(\Ic)$ and the ideals  $\ti
(\Ic,S)$ depend only on $\Ic$, not on a system of generators of
$\Ic$.

\medskip

 Let $S$ be a face of $\nc(\Ic)$ and $p_S\ag+q_S\bg =N_S$ be the equation of $l_S$, with $\gcd (p_S,q_S)=1$ as before. Then $\ti(\Ic,S)$ is of the form
 \begin{equation}\label{initialideal1}
 \ti(\Ic,S)=\left(x^{a_S}y^{b_S}F_{\Ic,S}(x^{q_S},y^{p_S})\right)
 \end{equation}
 or
 \begin{equation}\label{initialideal2}
 \ti(\Ic,S)=x^{a_S}y^{b_S}F_{\Ic,S}(x^{q_S},y^{p_S})\big(k_1(x^{q_S},y^{p_S}),\cdots,k_s(x^{q_S},y^{p_S})\big)
 \end{equation}
 with $s \geq 2$,
 where $F_{\Ic,S},k_1,\cdots,k_s$ are homogeneous polynomials, $F_{\Ic,S}$ is not divisible by $x$ or $y$  and $k_1,\cdots, k_s$ are coprime and of the same degree $d_S$. In the first case we put $d_S=0$.
 The polynomial $F_{\Ic,S}$, monic in $y$,  is called the {\it face polynomial} (it can be identically one). A face $S$ is called a {\it dicritical face} if $\ti(\Ic,S)$ is not a principal ideal.  Thus it is dicritical if and only if $d_S\geq 1$.

In the sequel we will also call the equation of the supporting line
of a face simply the equation of the face.

\medskip
 \subsection{Newton maps}

 \begin{definition}
 Let $(p,q)\in (\bn^*)^2$ with $\gcd (p,q)=1$. Take $(p',q')\in \bn^2$ such that $pp'-qq'=1$. Let $\mu\in \bc^*$. Define
$$ \begin{matrix}
 &\sg_{(p,q,\mu)}&:&\bc[[x,y]]&\longrightarrow&\bc[[x_1,y_1]]\\
 &&&f(x,y)&\mapsto&f(\mu^{q'}x_1^p,x_1^q(y_1+\mu^{p'})) .
 \end{matrix}$$
 We say that the map $\sg_{(p,q,\mu)}$ is a {\it Newton map}.
  \end{definition}

 The numbers $(p',q')$ are introduced only to avoid taking roots of complex numbers.
In the sequel we will always assume that $p'\leq q$ and $q'<p$. This will make procedures canonical.

Let $\Ic=(f_1,\cdots, f_r)$ be a non-trivial ideal in $\bc[[x,y]]$. Let $\sg_{(p,q,\mu)}$ be a Newton map. We denote by
$\sg_{(p,q,\mu)}(\Ic)$ the ideal in $\bc[[x_1,y_1]]$  generated by the $\sg_{(p,q,\mu)}(f_i)$ for $i=1,\cdots, r$. Since a Newton map is a ring homomorphism, this ideal does not depend on the choice of the generators of $\Ic$.

\begin{lemma}\label{Newton-algbis}\cite[Lemma 2.0]{CV}
Let $\Ic$ be a non-trivial ideal in $ \bc[[x,y]]$ and
$\sg_{(p,q,\mu)}(\Ic)=\Ic_1$.
\begin{enumerate}
\item If there does not exist a face $S$ of $\nc(\Ic)$ whose supporting line has equation $p\ag +q\bg=N$ with $N\in \bn$,  then the ideal $\Ic_1$ is principal,
generated by a power of $x_1$.
\item If there exists a face  $S$ of $\nc(\Ic)$ whose supporting line has equation $p\ag +q\bg =N_0$ for
  some $N_0 \in \bn$,  and if $F_{\Ic, S}(1,\mu)\neq 0$, then $\Ic_1=(x_1^{N_0})$.
\item If there exists a face $S$ of $\nc(\Ic)$ whose supporting line has equation $p\ag +q\bg =N_0$ for
  some $N_0 \in \bn$,  and if $F_{\Ic, S}(1,\mu)= 0$, then $\Ic_1=(x_1^{N_0})\Ic _1'$ and the height of the Newton polygon of $\Ic_1$ is less than or equal to the multiplicity of $\mu$ as root of $F_{\Ic, S}(1,X)$,
    \end{enumerate}
  \end{lemma}

\medskip
 \subsection{Newton algorithm}

Given an ideal $\Ic$ in $\bc[[x,y]]$ and a Newton map
$\sg_{(p,q,\mu)}$, we denote by $\Ic_{\sg}$ the ideal
$\sg_{(p,q,\mu)}(\Ic)$. Consider a sequence $\Sg_n=(\sg_1,\cdots,
\sg_n)$ of length $n$ of Newton maps. We define $\Ic _{\Sg_n}$ by
induction:
$$\Ic_{\Sg_1}=\Ic_{\sg_1},\cdots, \Ic_{\Sg_i}=(\Ic_{\Sg_{i-1}})_{\sg_i},\cdots,\Ic_{\Sg_n}=(\Ic_{\Sg_{n-1}})_{\sg_n}  .$$

\medskip
\begin{theorem}\label{Newtonalg}\cite[Theorem 2.10]{CV}
Let $\Ic$ be a non-trivial ideal in $\bc[[x,y]]$.  There exists an
integer $n_0$ such that, for any sequence $\Sg_n=(\sg_1,\cdots,
\sg_n)$ of Newton maps of length at least $ n_0$, the ideal $\Ic
_{\Sg_n}$ is principal, generated by $x^k(y+h(x))^{\nu}$ with $h\in
x\bc [[x]]$ and $(k,\nu)\in \bn\times \bn$.
\end{theorem}

\noindent
{\bf Example 1.}
We consider in $\bc[[x,y]]$ the ideal
$$\Ic=\left(y^4(y+x)(y^2-3x),((y+x)^3+x^8)(y^2-3x)\right).$$
Its Newton polygon is given in Figure 1.

  \begin{figure}[ht]
 \begin{center}
\includegraphics{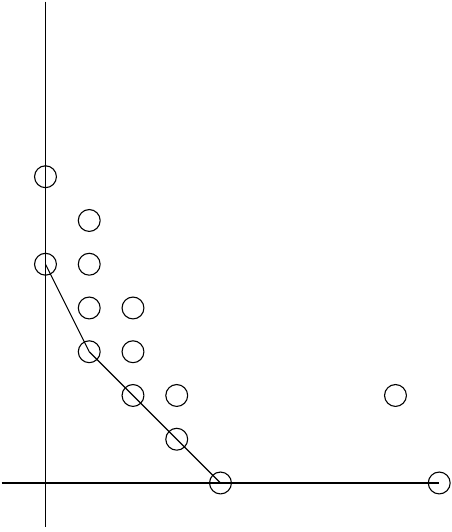}
\caption{}
 \end{center}
 \end{figure}

The faces $S_1$ and $S_2$ have supporting lines with equations $2\alpha+\beta=5$ and $\alpha+\beta=4$, respectively. The initial ideals of $\Ic$ with respect to these segments are
$$\ti (\Ic,S_1) = \left(y^3(y^2-3x)\right) \qquad\text{and}\qquad \ti (\Ic,S_2)= \left(x(y+x)^3\right).$$
Both segments are not dicritical; their face polynomials are $F_{\Ic,S_1}=y^2-3x$ and $F_{\Ic,S_2}=(y+x)^3$, respectively.
We first consider the Newton map $\sg_{(p,q,\mu)}=\sg_{(2,1,3)}$ associated to $S_1$ and $\mu=3$. It is given by the substitution
$$x=3x_1^2, \qquad\qquad y=x_1(y_1+3).$$
The image ideal $\Ic_1$ is given by
$$
\begin{aligned}
\Ic_1 &= \left(x_1^4 (y_1+3)^4(x_1y_1+3x_1+3x_1^2)(x_1^2(y_1+3)^2-9x_1^2)\right. ,\\
 &\qquad\qquad\qquad\qquad\qquad\qquad\left. ((x_1y_1+3x_1+3x_1^2)^3+3^8x_1^{16})(x_1^2(y_1+3)^2-9x_1^2)\right)\\
&=\left(x_1^7(y_1^2+6y_1),x_1^5(y_1^2+6y_1)\right)\\
&= (x_1^5y_1).
\end{aligned}
$$
It is a monomial ideal, hence we stop the procedure for $S_1$.

Next we consider the Newton map $\sg_{(p,q,\mu)}=\sg_{(1,1,-1)}$ associated to $S_2$ and $\mu=-1$. It is given by the substitution
$$x=x_1, \qquad\qquad y=x_1(y_1-1).$$
The image ideal $\Ic_1$ is given by
$$
\begin{aligned}
\Ic_1 &= \left(x_1^4(y_1-1)^4 x_1y_1(x_1^2(y_1-1)^2-3x_1), (x_1^3y_1^3+x_1^8)(x_1^2(y_1-1)^2-3x_1)\right)\\
&=\left(x_1^6y_1,x_1^4(y_1^3+x_1^5)\right)\\
&= x_1^4(x_1^2y_1,y_1^3+x_1^5).
\end{aligned}
$$
Its Newton polygon is given in Figure 2.

  \begin{figure}[ht]
 \begin{center}
\includegraphics{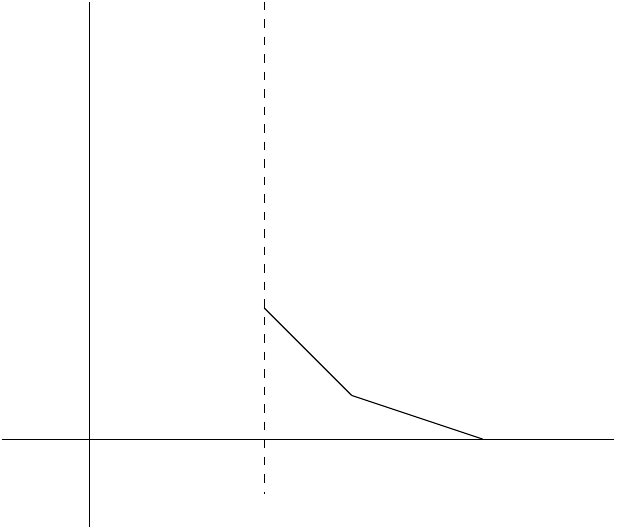}
\caption{}
 \end{center}
 \end{figure}

The faces $S'_1$ and $S'_2$ have supporting lines with equations $\alpha+\beta=7$ and $\alpha+3\beta=9$, respectively. The initial ideals of $\Ic_1$ with respect to these segments are
$$\ti (\Ic_1,S'_1) = x_1^4y_1(x_1^2,y_1^2) \qquad\text{and}\qquad \ti (\Ic_1,S'_2)= x_1^6(y_1,x_1^3).$$
Both segments are dicritical and have constant face polynomial; their degrees are $d_{S'_1}=2$ and $d_{S'_2}=1$.

We continue with the Newton maps $\sg_{(1,1,\mu)}$ associated to
$S'_1$ and $\sg_{(1,3,\mu)}$ associated to $S'_2$, where $\mu$ is
arbitrary in $\bc^*$. These lead to monomial ideals, hence we stop
the procedure.

\medskip
It should be clear that more generally, when a face $S$ of a
Newton polygon is dicritical, the associated Newton map
$\sg_{(p,q,\mu)}$ induces a monomial ideal for all $\mu\in \bc^*$
that are not roots of the face polynomial $F_{\Ic,S}$. Therefore it
is not necessary to compute such Newton maps explicitly.

\bigskip
In \cite{CV} we introduced a notion of depth of an ideal: the least
possible value of $n_0$ in Theorem \ref{Newtonalg}. This can also be
formulated as follows.

\medskip
\begin{definition}\label{def-depth} Let $\Ic$ be a non-trivial ideal in $\bc[[x,y]]$.
We define the {\it depth} of $\Ic$, denoted by $d(\Ic)$, by
induction. If $\Ic$ is principal, generated by $x^k(y+h(x))^{\nu}$
with $h\in x\bc[[x]]$ and $(k,\nu)\in \bn\times \bn$, we say that
its depth is $0$. Otherwise, we define
$$d(\Ic)=\max d(\Ic_{\sg})+1 ,$$
where the maximum is taken over all possible Newton maps.
\end{definition}

\medskip
\section{Computation of the motivic zeta function of an ideal}

\medskip
In this section we describe an efficient algorithm for the motivic
zeta function of an ideal in $\bc[[x,y]]$ in terms of the Newton
algorithm. The method is the one used in \cite{ACLM1}\cite{ACLM2}
for quasi-ordinary hypersurface singularities.

First we introduce the ingredients needed to define the motivic zeta
function.

\medskip
 Let $\Gc$ be the Grothendieck ring of algebraic varieties
over $\bc$. It is the abelian group generated by the symbols $[S]$
where $S$ is an algebraic variety over $\bc$, with the relations
$[S]=[S']$ when $S$ and $S'$ are isomorphic, and
$$[S]=[S\setminus S']+[S']$$
when $S'$ is closed in $S$. The product is defined by
$$[S]\cdot [S']=[S\times S'].$$
Let $\bl=[\ba_{\bc}^1]$ and $\Gc_{\text{loc}}=\Gc[\bl ^{-1}]$.

\smallskip
We recall briefly the notions of jet scheme and arc scheme. A
reference in the context of  arbitrary schemes is \cite{NS}. The
functor $\cdot \times_{\spec \bc} \spec \bc[t]/(t^{n+1})$ on the
category of $\bc$-schemes has a right adjoint, denoted by $\cl_n$.
We call $\cl_n(X)$ the $n$-jet scheme of $X$ and its closed points
$n$-jets on $X$. For $m\geq n$, the closed immersions $\spec
\bc[t]/(t^{n+1}) \hookrightarrow \spec \bc[t]/(t^{m+1})$, defined by
reduction modulo $t^{n+1}$, induce canonical (projection) morphisms
$\pi^m_n: \cl_m(X)\to\cl_n(X)$. These morphisms being affine, the
projective limit
$$\cl(X):= \lim_{\longleftarrow} \cl _n(X)$$ exists as a
$\bc$-scheme; it is called the arc scheme of $X$ and its
$\bc$-points are called arcs on $X$. Denote by $\pi _n: \cl (X)
\longrightarrow \cl_n(X)$ the natural projection morphisms.

From now on we take $X=\spec \bc[[x,y]]$ and we consider only those
arcs and jets attached at the origin, that is, mapped by $\pi_0$ or
$\pi^n_0$ to the origin in $X=\cl_0(X)$. We denote these schemes
(and also the sets of their $\bc$-points) by $\cl^0(\bc^2)$ and
$\cl^0_n(\bc^2)$, respectively. Note that they are isomorphic to
$\spec \bc [a_i, b_i]_{i\in \bn^*}$ and $\spec \bc [a_1,
b_1,\dots,a_n,b_n]$, respectively. For example an $n$-jet is denoted
as $(a_1 t + a_2 t^2 + \cdots + a_n t^n, b_1 t + b_2 t^2 + \cdots +
b_n t^n)$.

\smallskip

A {\em cylinder} in
 $\cl^0 (\bc ^2)$ is a subset of the form $A=(\pi_n)^{-1}(C)$, with
 $C$ a constructible subset of $\cl^0_n(\bc^2)$, for some $n \in
 \bn$. For such a set $A$ we have that $[\pi_m(A)]\bl^{-2m}=[\pi_n(A)] \bl^{-2n}=[C]\bl^{-2n}$ for all $m\geq n$, and one defines its {\em motivic
measure} $\mu(A)$ as this value:
$$\mu(A)=[\pi_n(A)]\bl^{-2n} \in \Gc_{\text{loc}}.$$

\bigskip
Let $\Ic$ be a non-trivial ideal in $\bc[[x,y]]$. We write
$\Ic=(x^N)\Ic'$, where $\Ic'$ is not divisible by $x$. Let $\og$ be
a regular differential $2$-form of the form
$$\og=x^{\nu-1}dx \wedge dy.$$
We assume that if $N=0$, then $\nu=1$.

 For $\phi\in \cl^0(\bc^2)$ we define
$$\text{ord}_{\phi } \Ic=\min \{\text{ord}(f \circ \phi ) \mid f \in \Ic\}.$$
For $m\in\bn$ and $n\in \bn^*$ we consider
 $$V_{n,m}=\{\phi\in
\cl^0(\bc^2) \mid \text{ord}_\phi \Ic=n,\ \text{ord}(\og \circ \phi
)=m\}.$$
 Note that the motivic measure of this collection of arcs is completely determined by the
 corresponding collection of $n$-jets
$$\pi_n(V_{n,m})=\{\phi\in \cl^0_n(\bc^2) \mid \text{ord}_\phi \Ic=n,\
\text{ord}(\og \circ \phi )=m\}.$$
 The motivic zeta function of
$\Ic$ and $\omega$ is
$$\zg(\Ic,\og)(T)=\sum_{n\geq 1}\big(\sum_{m\geq 0}\mu(V_{n,m})\bl^{-m}\big)T^n \in \Gc_{\text{loc}}[[T]]. $$
 The hypothesis on $\omega$ ensures that the sum over $m$ is finite.

The motivic zeta function was originally introduced for a principal
ideal $\Ic$ (in a polynomial ring) \cite{DL1}. The generalization to
arbitrary ideals is straightforward (see for instance \cite{VZ}).
Incorporating also a differential form is natural in the process of
studying or computing a motivic zeta function, see e.g.
\cite{ACLM2}\cite{V2}.

An important result in \cite{DL1} is that motivic zeta functions are
in fact rational functions (in $T$); this is proven by providing a
formula in terms of a chosen embedded resolution of singularities.
The proof extends to zeta functions associated to a general ideal
and a differential form. However, in that formula generally a large
amount of false candidate poles occur. In dimension 2 the results of
\cite{V1} provide in principle a compact formula with only actual
poles in the denominator, but the computation is a combinatorial
disaster.

\bigskip
We now start deriving an efficient algorithmic way to compute
$\zg(\Ic,\og)(T)$, based on the Newton algorithm, yielding in
particular an \lq almost minimal\rq\ denominator.  (In the next
section we explain a slight modification resulting in the optimal
denominator.)

 First we consider the case
where the ideal $\Ic$ is principal, generated by a monomial:
 $$\Ic=(x^{N_1}y^{N_2}).$$
 Let $$
 \phi(t)=\left \{
\begin{matrix}
 &x&=& c_1t^{k_1}+\cdots+a_nt^n + \cdots\\
 &y&=& c_2t^{k_2}+\cdots+b_nt^n + \cdots
 \end{matrix}
  \right . $$
be an arc in $V_{n,m}$, where $c_1,c_2\in \bc^*$.
  We have
  $$\text{ord}_\phi \Ic=k_1N_1+k_2N_2=n,$$
  $$\text{ord}(\og \circ \phi )=(\nu-1)k_1=m,$$
 $$ \mu(V_{n,m})=[\pi_n(V_{n,m})]\bl^{-2n}.$$
 Then
 $$
 \aligned \zg(\Ic,\og)(T)&=(\bl-1)^2\sum_{k_1\geq1,k_2\geq 1}\bl^{n-k_1}\bl^{n-k_2}\bl^{-2n}\bl^{-(\nu-1)k_1}T^n \\
 &=(\bl-1)^2\sum_{k_1\geq1,k_2\geq 1}\bl^{-\nu k_1-k_2}T^{k_1N_1+k_2N_2} \\
 &=(\bl-1)^2\frac{\bl^{-\nu}T^{N_1}}{1-\bl^{-\nu}T^{N_1}}\frac{\bl^{-1}T^{N_2}}{1-\bl^{-1}T^{N_2}}.
 \endaligned
 $$

\bigskip
\begin{remark}
Above we tacitly assume that $N_1$ or $N_2$ can be zero. In that
case a  factor of the form $1-\bl^{-\nu}$ appears in the denominator
of the formula. More generally, in the rest of this section,
formulas for $\zg(\Ic,\og)(T)$ as a rational function in $T$ have
coefficients in a further localization of $\Gc_{\text{loc}}$, namely
with respect to the elements $1-\bl^{-a}, a\in \bz^*$. This is
standard in such computations.
\end{remark}

\bigskip
To compute the zeta function of an arbitrary ideal,  we  use
induction on the depth of the ideal. If the ideal has depth $0$, we
already computed its zeta function (since it does not depend on the
chosen coordinate system). For an ideal $\Ic$ of depth at least $1$,
we derive a formula computing the zeta function of $\Ic$ in terms of
the zeta functions of the ideals $\Ic_{\sg}$.  First we collect the
necessary data and notation.

\begin{definition}\label{def-cones} \rm
Let $\Ic$ be a non-trivial ideal in $\bc[[x,y]]$ and
$\omega=x^{\nu-1}dx\wedge dy$ with $\nu\in\bn^*$. If $\Ic$ is not
divisible by $x$, we assume that $\nu=1$.

Consider the Newton polygon of $\Ic$ and its dual in $\br_+^2$. Here
a face $S$ of the Newton polygon of $\Ic$ with equation
$p\ag+q\bg=N$ corresponds in the dual to a half-line going through
the origin with direction $(p,q)$. This induces a decomposition of
$\br_+^2$ into disjoint cones (of dimension $1$ or $2$).

(1) We denote by $\Dg$ such a $2$-dimensional cone, and let it be
generated by $(p_1,q_1)$ and $(p_2,q_2)$. Then we denote also
 $$\pc_{\Dg}=\{(i,j)\in \bn^2\mid i=\mu_1p_1+\mu_2p_2,j=\mu_1q_1+\mu_2q_2, (\mu_1,\mu_2)\in \bq^2, 0<\mu_1\leq1,0<\mu_2\leq1\}$$
and
 $$D_{\Dg}=\sum_{(i,j)\in \pc_{\Dg}}\bl^{-(\nu i+j)}T^{ia+jb}.$$

(2) We denote by $L$ such a $1$-dimensional cone, and let it be
generated by $(p,q)$.
 \end{definition}

\noindent Note that in (1) the set of points with integer
coordinates in the cone $\Dg$ is precisely
 $$\{(i,j)\in \bn^2\mid i=i_0+k_1p_1+k_2p_2,j=j_0+k_1q_1+k_2q_2,(i_0,j_0)\in \pc_{\Dg},(k_1,k_2)\in \bn^2\}.$$

\bigskip
 \noindent We study the contributions of elements
$(k_1,k_2)\in \bn^2$, depending on which cone they belong to.

 \begin{enumerate}
 \item {\em Assume that $(k_1,k_2)$ lies in a cone $\Dg$ generated by $(p_1,q_1),(p_2,q_2)$.} Let $(a,b)$ be the intersection point of the two faces of the Newton polygon whose supporting lines have equations
$ p_1 \ag+q_1\bg =N_1$ and  $ p_2 \ag+q_2\bg =N_2$.  Let $\phi \in
\cl^0 (\bc^2):$
 $$
 \phi (t)=\left \{
\begin{matrix}
 &x&=& c_1t^{k_1}+\cdots+a_nt^n+\cdots\\
 &y&=& c_2t^{k_2}+\cdots+b_nt^n+\cdots
 \end{matrix}
  \right . $$
where $c_1,c_2\in \bc^*$, such that $\text{ord}_\phi \Ic=n$. We have
   $$\text{ord}_\phi \Ic=k_1a+k_2b=n,$$
  $$\text{ord}(\og \circ \phi )=(\nu-1)k_1.$$
  The contribution of $\Delta$ to $\zg(\Ic,\og)(T)$ is

  $$
\aligned  &(\bl-1)^2\sum_{(k_1,k_2)\in \Dg\cap \bn^2}\bl^{n-k_1}\bl^{n-k_2}\bl^{-2n}\bl^{-(\nu-1)k_1}T^n\\
   =\ &(\bl-1)^2\sum_{(k_1,k_2)\in \Dg\cap \bn^2}\bl^{-\nu
   k_1-k_2}T^{k_1a+k_2b}\\
=\ &(\bl-1)^2 D_{\Dg}  \sum_{(n_1,n_2)\in \bn^2} \bl ^{-[(p_1\nu+q_1)n_1+(p_2\nu+q_2)n_2]}T^{(p_1a+q_1b)n_1+(p_2a+q_2b)n_2}\\
=\ &(\bl-1)^2 D_{\Dg}
\frac{1}{(1-\bl^{-(p_1\nu+q_1)}T^{N_1})(1-\bl^{-(p_2\nu+q_2)}T^{N_2})}.
\endaligned$$

\item {\em Assume that $(k_1,k_2)$ lies on a line $L$, generated by $(p,q)$.}
Then there exists $k\in \bn^*$ such that $k_1=pk, k_2=qk$. Let $S$
be the corresponding face of the Newton polygon, with equation
$p\alpha+q\beta=N$, and let $F_{S}$ be its face polynomial.  Let
$\phi \in \cl^0 (\bc^2):$
$$
 \phi (t)=\left \{
\begin{matrix}
 &x&=& c_1t^{pk}+\cdots+a_nt^n+\cdots\\
 &y&=& c_2t^{qk}+\cdots+b_nt^n+\cdots
 \end{matrix}
  \right . $$
  where $c_1,c_2\in \bc^*$, such that $\text{ord}_\phi \Ic=n$.
We have two cases.
\smallskip
\begin{enumerate}
\item For any root $\mu$ of $F_{S}$, we have $c_2^p-\mu c_1^q\neq 0$. In this case,
  $$\text{ord}_\phi \Ic=Nk,$$
  $$\text{ord}(\og \circ \phi )=p(\nu-1)k.$$
Denoting by $r$ the number of distinct roots of $F_{S}$, we compute
this contribution of $L$ to $\zg(\Ic,\og)(T)$. It is equal to
$$((\bl-1)^2 - C) \sum_{k\in \bn^*}\bl^{n-pk+n-qk-p(\nu-1)k}\bl^{-2n}T^{Nk},$$
where $C$ is the class in the Grothendieck ring of
$$
\{(c_1,c_2)\in (\bc^*)^2 \mid c_2^p-\mu_j c_1^q= 0 \ \text{for some
root} \ \mu_j \ \text{of}\ F_S\}.
$$
Hence this contribution is
$$\aligned
&((\bl-1)^2-r(\bl-1))\sum_{k\in \bn^*}\bl^{-(p\nu+q)k}T^{Nk}\\
=\ & (\bl-1)(\bl-r-1)\frac{\bl ^{-(p\nu+q)}T^N}{1-\bl^{- (p\nu
+q)}T^N}.
\endaligned $$

\smallskip
\item There exists a root $\mu$ of $F_{l_S}$ such that $c_2^p-\mu c_1^q= 0$. Fix such a root $\mu$.
We need to compute
 $$\sum_{n\geq 1}\sum_{k\geq 1}[X_{n,(\nu-1)pk}]\bl^{-2n}\bl^{-(\nu-1)pk}T^n,$$
where
  $X_{n,(\nu-1)pk}$ is the variety of  $n$-jets $\phi:$
$$ \phi (t)=\left \{
\begin{matrix}
 &x&=& a_{pk}t^{pk}+\cdots+a_nt^n\\
 &y&=& b_{qk}t^{qk}+\cdots+b_nt^n
 \end{matrix}
 \right.$$
such that $\text{ord}_{\phi}\Ic =n$, $\text{ord}( \og \circ \phi)=(\nu-1)pk$ and $b_{qk}^p -\mu a_{pk}^q=0$.

 Denote by $\overline{X}_{n,(\nu-1)pk}$ the variety of $n$-jets $\overline{ \phi}:$
$$\overline{ \phi} (t)=\left \{
\begin{matrix}
 &\overline{\phi}_1&=& a'_0+a'_1t+\cdots+a'_nt^n\\
 &\overline{\phi}_2&=& b'_0+b'_1t+\cdots+b'_nt^n
 \end{matrix}
 \right.$$
such that $\text{ord}_{(t^{pk}\overline{\phi}_1,t^{qk}\overline{\phi}_2)}\Ic =n$,  $\text{ord}( \og \circ (t^{pk}\overline{\phi}_1,t^{qk}\overline{\phi}_2))=(\nu-1)pk$ and $(b'_{0})^{p} -\mu (a'_{0})^{q}=0$.
 We have
 $$[X_{n,(\nu-1)pk}]=[\overline{X}_{n,(\nu-1)pk}]\mathbb{L}^{-pk-qk}.$$

\smallskip
Next we consider the appropriate varieties of jets related to the Newton map $\sigma=\sigma_{(p,q,\mu)}$.
Let $\og_1=x_1^{p\nu+q-1}dx_1\wedge dy_1$.
 Denote by $\overline{X}^{\sg}_{n,(p\nu+q-1)k}$ the variety of $n$-jets $\overline{ \psi}:$
$$\overline{ \psi} (t)=\left \{
\begin{matrix}
 &\overline{\psi}_1&=& r_0+r_1t+\cdots+r_nt^n\\
 &\overline{\psi}_2&=& s_1t+\cdots+s_nt^n
 \end{matrix}
 \right.$$
such that $r_0\neq 0$,  $\text{ord}_{(t^k\overline{\psi}_1,\overline{\psi}_2)}\Ic_{\sg} =n$ and $\text{ord}( \og_1 \circ (t^k\overline{\psi}_1,\overline{\psi}_2))=(p\nu+q-1)k$.
 Finally we denote by $X^{\sg}_{n,(p\nu +q-1)k}$ the variety of $n$-jets $\psi:$
$$ \left\{ \begin{matrix}
 &{\psi}_1&=& r'_kt^k+\cdots+r'_nt^n\\
 &{\psi}_2&=& s'_1t+\cdots+s'_nt^n
 \end{matrix}
 \right.$$
such that $r'_k\neq 0$,  $\text{ord}_{({\psi}_1,{\psi}_2)}\Ic_{\sg} =n$ and $\text{ord}( \og_1 \circ ({\psi}_1,{\psi}_2))=(p\nu+q-1)k$.
  We have
  $$[\overline{X}^{\sg}_{n,(p\nu+q-1)k}]=[X^{\sg}_{n,(p\nu +q-1)k}]\mathbb{L}^k.$$

We construct a map $S$ from $\overline{X}^{\sg}_{n,(p\nu+q-1)k}$ to
$\overline{X}_{n,(\nu-1)pk}$ in order to compare
$[\overline{X}_{n,(\nu-1)pk}]$ and
$[\overline{X}^{\sg}_{n,(p\nu+q-1)k}]$.

For $\overline{ \psi}\in \overline{X}_{n,(p\nu +q-1)k}$ given by
$$\overline{ \psi} (t)=\left \{
\begin{matrix}
 &\psi_1&=&r_0+r_1t+\cdots +r_nt^n\\
 &\psi_2&=&s_1t+\cdots+s_nt^n
 \end{matrix}
 \right.$$
we define
 $$\textit{S}(\overline{ \psi}) (t)=\left \{
\begin{matrix}
 &\overline{\phi}_1&=&\mu^{q'}\overline{\psi}_1^p \mod t^{n+1} \\
 &\overline{\phi}_2&=&\overline{\psi}_1^q(\overline{\psi}_2+\mu^{p'}) \mod t^{n+1}.
 \end{matrix}
 \right.$$

Then one easily verifies that
 $$\text{ord}_{(t^k\overline{\psi}_1,\overline{\psi}_2)}\Ic_{\sg} =n=\text{ord}_{(t^{pk}\overline{\phi}_1,t^{qk}\overline{\phi}_2)}\Ic.$$

The constant term of $\overline{\phi}_1$ is $a=\mu^{q'}r_0^p$ and
the constant term of $\overline{\phi}_2$ is $b=\mu^{p'}r_0^q$. Hence
we have $b^p=\mu a^q$. Note also that $S$ is clearly a morphism.

 \begin{lemma}
The map $\textit{S}$ defines a isomorphism between
$\overline{X}^{\sg}_{n,(p\nu+q-1)k}$ and
 $\overline{X}_{n,(\nu-1)pk}$.
  \end{lemma}
 \begin{proof}
We construct an inverse map. Take
 $$\overline{ \phi} (t)=\left \{
\begin{matrix}
 &\overline{\phi}_1(t)&=& a'_0+a'_1t+\cdots+a'_nt^n=a'_0\tilde{\phi}_1(t)\\
 &\overline{\phi}_2(t)&=& b'_0+b'_1t+\cdots+b'_nt^n=b'_0\tilde{\phi}_2(t)
 \end{matrix}
 \right.$$

 in $ \overline{X}_{n,(\nu-1)pk}$. Then

  $$\overline{ \psi} (t)=\left \{
\begin{matrix}
 &\overline{\psi}_1&=& a_0^{'p'}/b_0^{'q'}(\tilde{\phi}_1)^{1/p} \mod t^{n+1}\\
 &\overline{\psi}_2&=&
 -\mu^{p'}+\mu^{p'}\tilde{\phi}_2/(\tilde{\phi}_1)^{q/p} \mod t^{n+1}
 \end{matrix}
 \right.$$

belongs to $\overline{X}^{\sg}_{n,(p\nu+q-1)k}$. Here
$(\tilde{\phi}_1)^{1/p}$ is well-defined and given by the standard
expansion since the constant term of $\tilde{\phi}_1$ is $1$.
One easily verifies that this map
$\overline{\phi}\mapsto\overline{\psi}$ is indeed the inverse of
$S$.

 \end{proof}



\medskip
With these preparations we can compute
$$\aligned
&\sum_{n\geq 1}\sum_{k\geq 1}[X_{n,(\nu-1)pk}]\bl^{-(\nu-1)pk}\bl^{-2n}T^n\\
=& \sum_{n\geq 1}\sum_{k\geq 1}[\overline{X}_{n,(\nu-1)pk}]\bl^{-qk} \bl^{-\nu pk}\bl^{-2n}T^n\\
=& \sum_{n\geq 1}\sum_{k\geq 1}[ \overline{X}^{\sg}_{n,(p\nu+q-1)k}]\bl^{-(q+p\nu)k}\bl^{-2n}T^n\\
=&  \sum_{n\geq 1}\sum_{k\geq 1}[ X^{\sg}_{n,(p\nu+q-1)k}]\bl^{-(p\nu+q-1)k}\bl^{-2n}T^n \\
=&  \sum_{n\geq1}(\sum_{m}\mu(V_{n,m}^{\sg})\bl^{-m})T^n,
\endaligned$$
  where $V_{n,m}^{\sg}$ is the set of arcs $\phi \in \cl^0(\bc^2)$ such that
  $\text{ord}_{\phi}\Ic_\sigma=n$ and $\text{ord} (\og_1\circ \phi)=m$. This last expression is precisely
$\zg(\Ic_\sg,\og_1)(T)$.

 \end{enumerate}
 \end{enumerate}
We have proven the following formula.

 \begin{theorem}\label{formulazetafunction}
Let $\Ic$ and $\omega$ be as in definition \ref{def-cones}. With the
notation introduced there, we have
   $$\zg(\Ic,\og)(T)=\sum_{\Delta} (\bl-1)^2 D_{\Dg}  \frac{1}{(1-\bl^{-(p_1\nu+q_1)}T^{N_1})(1-\bl^{-(p_2\nu+q_2)}T^{N_2})}$$
 $$ + \sum_L (\bl-1)(\bl-r-1)\frac{\bl ^{-(p\nu+q)}T^N}{1-\bl^{- (p\nu +q)}T^N}$$
 $$+\sum_L\sum_{\mu}\zg(\Ic_{\sg_{(p,q,\mu)}},\og_1)(T),$$
where $\mu$ ranges over the roots of the face polynomial of the face
corresponding to the cone $L$, and $\og_1=x_1^{p\nu+q-1}dx_1\wedge
dy_1$.
 \end{theorem}

\medskip
\noindent {\bf Example 1 {\rm (continued)}.} The first Newton
polygon gives rise in the dual plane to a decomposition of $\br_+^2$
in three ($2$-dimensional) cones and two lines. The cones are
$\Delta_0$ generated by $(1,0)$ and $(2,1)$, $\Delta_1$ generated by
$(2,1)$ and $(1,1)$, and $\Delta_2$ generated by $(1,1)$ and
$(0,1)$. The lines are $L_1$ generated by $(2,1)$, and $L_2$
generated by $(1,1)$.

 \begin{figure}[ht]
 \begin{center}
\includegraphics{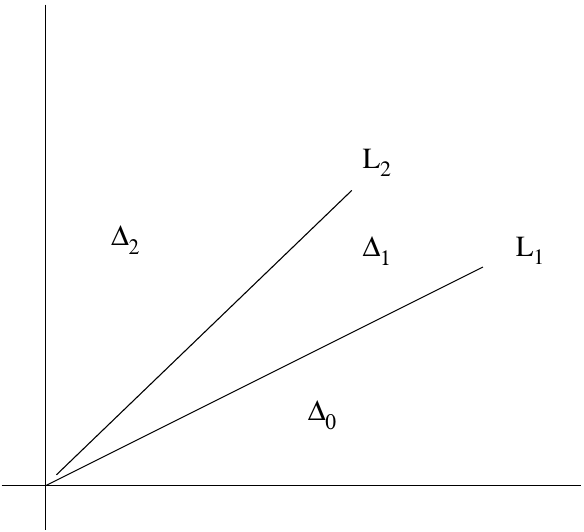}
\caption{}
 \end{center}
 \end{figure}

The contribution of the cones to the zeta function for $\Ic$ and $\omega=x^{\nu-1}dx\wedge dy$ is
$$
\aligned
(\mathbb{L}-1)^2\Big(\frac{\mathbb{L}^{-(3\nu+1)}T^5}{(1-\mathbb{L}^{-(2\nu+1)}T^5)(1-\mathbb{L}^{-\nu})}
&+
\frac{\mathbb{L}^{-(3\nu+2)}T^9}{(1-\mathbb{L}^{-(2\nu+1)}T^5)(1-\mathbb{L}^{-(\nu+1)}T^4)}
\\ &+
\frac{\mathbb{L}^{-(\nu+2)}T^4}{(1-\mathbb{L}^{-(\nu+1)}T^4)(1-\mathbb{L}^{-1})}\Big).
\endaligned$$
 The contribution of the lines is
 $$(\mathbb{L}-1)(\mathbb{L}-2)\Big(\frac{\mathbb{L}^{-(2\nu+1)}T^5}{(1-\mathbb{L}^{-(2\nu+1)}T^5)}+
 \frac{\mathbb{L}^{-(\nu+1)}T^4}{(1-\mathbb{L}^{-(\nu+1}T^4)}\Big).$$
We have considered the Newton map $\sg_{(2,1,3)}$ given by
  $$x=3x_1^2, \qquad\qquad y=x_1(y_1+3),$$
with image ideal $\Ic_1=(x_1^5y_1)$. We have $\omega_1=x_1^{2\nu}dx_1\wedge dy_1 $.

The contribution of the ideal $\Ic_1$ is
$$(\mathbb{L}-1)^2\frac{\mathbb{L}^{-(2\nu+1)}T^5}{(1-\mathbb{L}^{-(2\nu+1)}T^5)}\frac{\mathbb{L}^{-1}T}{(1-\mathbb{L}^{-1}T)}.$$
Next we consider the Newton map $\sg_{(p,q,\mu)}=\sg_{(1,1,-1)}$ associated to $S_2$ and $\mu=-1$. It is given by the substitution
$$x=x_1, \qquad\qquad y=x_1(y_1-1),$$
with image ideal given by
$$
\Ic_1 = x_1^4(x_1^2y_1,y_1^3+x_1^5).
$$
Its Newton polygon is given in Figure 2.  We have $\omega _1=x_1^{\nu}dx_1\wedge dy_1$.

We have again three ($2$-dimensional) cones in the dual space:
$\Delta'_0$ generated by $(1,0)$ and $(1,1)$, $\Delta'_1$ generated
by $(1,1)$ and $(1,3)$, and $\Delta'_2$ generated by $(1,3)$ and
$(0,1)$. The contribution of theses cones to the motivic zeta
function is
$$\aligned
(\mathbb{L}-1)^2\Big(\frac{\mathbb{L}^{-(2\nu +3)}T^{11}}{(1-\mathbb{L}^{-(\nu +2)}T^7)(1-\mathbb{L}^{-(\nu +1)}T^4)}&+\frac{(\mathbb{L}^{-(\nu +3)}T^8+\mathbb{L}^{-(2\nu+6)}T^{16}}{(1-\mathbb{L}^{-(\nu +2)}T^7)(1-\mathbb{L}^{-(\nu +4)}T^9)}\\
&+\frac{\mathbb{L}^{-(\nu +5)}T^9}{(1-\mathbb{L}^{-(\nu
+4)}T^9)(1-\mathbb{L}^{-1})}\Big).
\endaligned$$

The two faces are dicritical faces. The contribution of the two lines is
$$(\mathbb{L}-1)^2\Big(\frac{\mathbb{L}^{-(\nu+2)}T^7}{(1-\mathbb{L}^{-(\nu+2)}T^7)}+
 \frac{\mathbb{L}^{-(\nu+4)}T^9}{(1-\mathbb{L}^{-(\nu+4)}T^9)}\Big).$$

 The final calculation gives
$$\zeta(\Ic,\omega)(T)=\frac{P(\mathbb{L}^{-1},T)}{(1-\mathbb{L}^{-(\nu+4)}T^9))(1-\mathbb{L}^{-(\nu+2)}T^7)(1-\mathbb{L}^{-1}T)},$$
 where $P(\mathbb{L}^{-1},T)$ is a polynomial.

\bigskip
We observe that all faces of the successive Newton polygons
contribute to the computation of the motivic zeta function. A face
with equation $p\ag+q\bg =N$ gives rise to a factor in the
denominators of the contributions of the form $(1-\mathbb{L}^{-(p
\nu +q)}T^N)$, where the differential form at that stage of the
Newton algorithm has the form $\omega=x^{\nu-1}dx\wedge dy$.
However, at the end of the computation not all of these factors
remain. We prove a general result concerning this phenomenon.

\begin{prop}\label{no-contribution}
 Take $\Ic$ and $\omega$ as in Theorem \ref{formulazetafunction}.
Let $S$ be a face of the Newton polygon of $\Ic$ with equation
$\ag+q\bg=N$ that hits the $x$-axis, which is not a dicritical face
and such that its face polynomial has exactly one root. Then
$(1-\mathbb{L}^{-(\nu +q)}T^N)$ does not appear in the denominator
of the motivic zeta function of $\Ic$ and $\omega$. The analogous
result holds for a face $p\ag +\bg=N$ that hits the $y$-axis, with
the same hypothesis.
\end{prop}

   \begin{figure}[ht]
 \begin{center}
\includegraphics{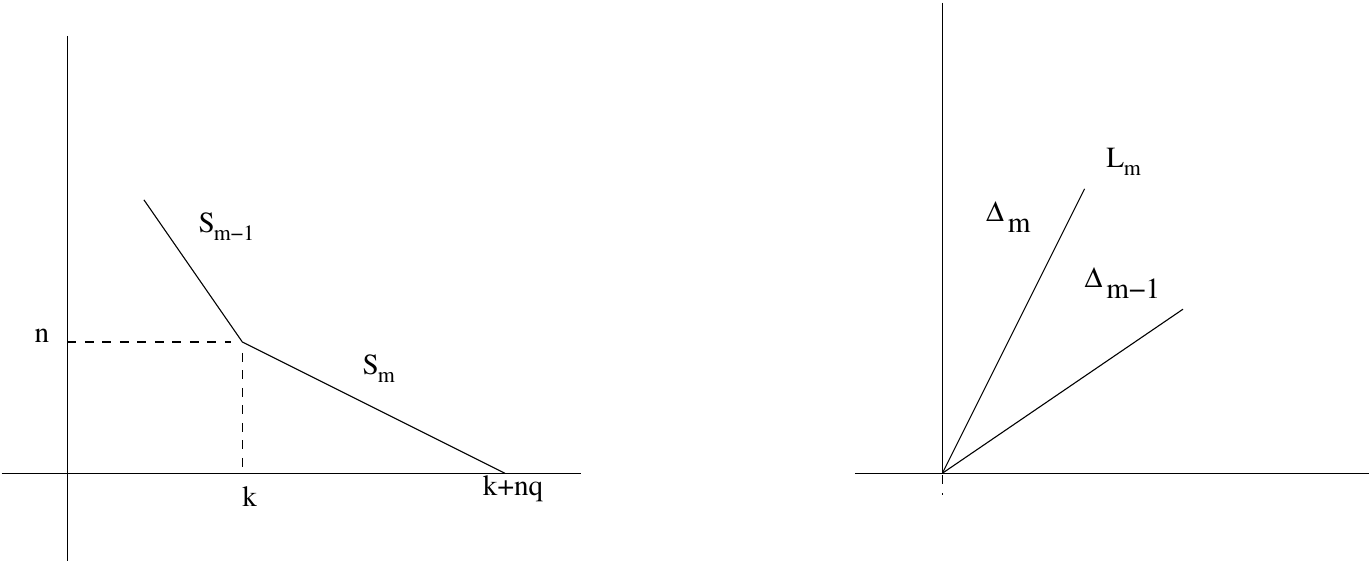}
 \caption{}
 \end{center}
\end{figure}

\begin{proof}
We assume the face $S=S_m$ has equation $\ag+q\bg=N$. Its
extremities have coordinates of the form $(k,n)$ and $(k+nq,0)$ with
$N=k+nq$. Writing the equation of the face $S_{m-1}$ as
$p_{m-1}\ag+q_{m-1}\bg=N_{m-1}$, we have
$N_{m-1}=p_{m-1}k+q_{m-1}n$. If $k=0$, the face $S_{m-1}$ has
equation $\alpha=0$.

There are four contributions of the face $S_m$ to the zeta function.
The contribution of $\Delta_m$ is
$$(\mathbb{L}-1)^2\frac{\mathbb{L}^{-(1+\nu+q)}T^N}{(1-\mathbb{L}^{-1})(1-\mathbb{L}^{-(\nu+q)}T^N)}$$
and the contribution of $L_m$ is
$$(\mathbb{L}-1)(\mathbb{L}-2)\frac{\mathbb{L}^{-(\nu+q)}T^N}{(1-\mathbb{L}^{-(\nu+q)}T^N)}.$$
The sum of these two contributions is
$$\aligned
\zeta _1(\Ic,\omega)(T)&=(\mathbb{L}-1)^2\frac{\mathbb{L}^{-(\nu+q)}T^N}{(1-\mathbb{L}^{-(\nu+q)}T^N)}\\
&=(\mathbb{L}-1)^2\Phi _1(x_1,x_2)\vert_{x_1=\mathbb{L}^{-\nu}T^k,x_2=\mathbb{L}^{-1}T^n},
\endaligned$$
where
 $$\Phi_1(x_1,x_2)=\frac{x_1x_2^q}{1-x_1x_2^q}.$$
The contribution of the cone $\Delta_{m-1}$  is
$$\aligned
\zeta _2(\Ic,\omega)(T)&=(\mathbb{L}-1)^2\frac{D_{\Delta_{m-1}}}{(1-\mathbb{L}^{-(\nu+q)}T^N)(1-\mathbb{L}^{-(p_{m-1}\nu+q_{m-1})}T^{N_{m-1}})}\\
&=(\mathbb{L}-1)^2\Phi _2(x_1,x_2)\vert_{x_1=\mathbb{L}^{-\nu}T^k,x_2=\mathbb{L}^{-1}T^n},
\endaligned$$
where $$\Phi_2(x_1,x_2)=\frac{D(x_1,x_2)}{(1-x_1x_2^q)(1-x_1^{p_{m-1}}x_2^{q_{m-1}})}$$
and
$$D(x_1,x_2)=\sum_{(i,j)\in \mathcal{P}_{\Delta_{m-1}}} x_1^i x_2^j.$$

The remaining part arises after applying the Newton map $\sigma_{(1,q,\mu)}:$
$$x=x_1,\quad\quad y=x_1^q(y_1+\mu),$$
where $\mu$ is the root of $F_{S_m}$.
We have $\omega_1=x_1^{\nu+q-1}dx_1\wedge dy_1$.

   \begin{figure}[ht]
 \begin{center}
\includegraphics{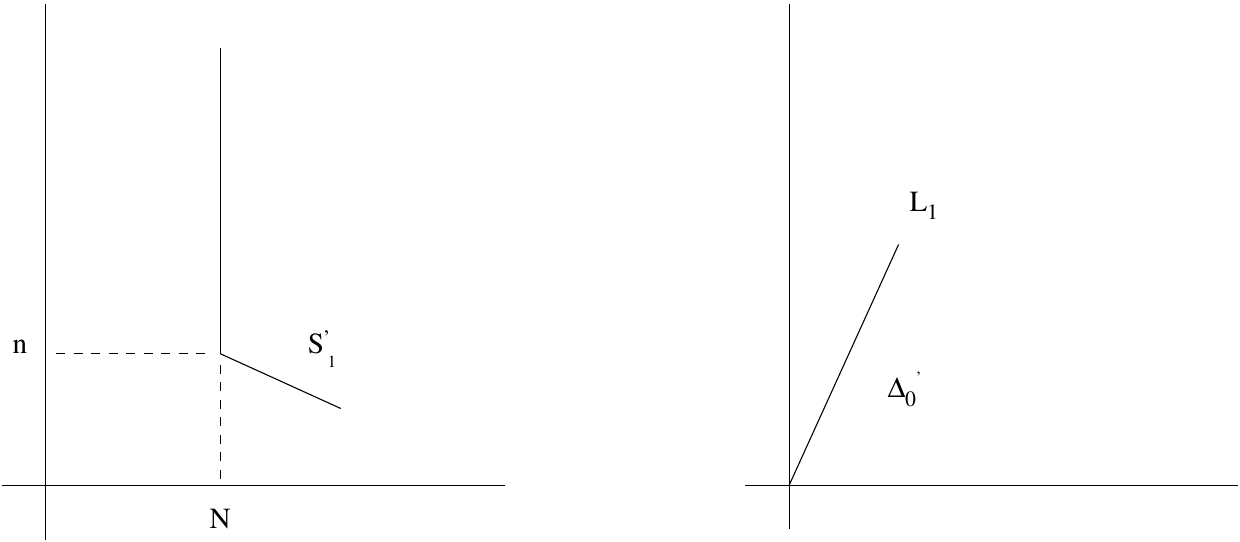}
 \caption{}
 \end{center}
\end{figure}

\noindent
 The contribution to the zeta function comes from the cone
$\Delta_0'$ for the new Newton polygon. Writing the equation of the
new face $S_1'$ as $p_1'\alpha +q_1' \beta =N_1'$, we have
$N_1'=p_1'N+q_1'n$. Then this contribution is
$$
\zeta _3(\Ic,\omega)(T)=
(\mathbb{L}-1)^2 \frac{D_{\Delta_0'}}{(1-\mathbb{L}^{-(\nu+q)}T^N)(1-\mathbb{L}^{-(p'_1(\nu+q)+q_1')}T^{N_1'})},$$
where
$$D_{\Delta_0'}=\sum_{(i,j)\in \mathcal{P}_{\Delta_0'} }(\mathbb{L}^{-(\nu+q)}T^N)^i (\mathbb{L}^{-1}T^n)^j.$$
We can write
$$(\mathbb{L}^{-(\nu+q)}T^N)^i (\mathbb{L}^{-1}T^n)^j=(\mathbb{L}^{-\nu}T^k)^i(\mathbb{L}^{-1}T^n)^{j+qi}$$
and similarly
$$1-\mathbb{L}^{-(\nu+q)}T^N=1-\mathbb{L}^{-\nu}T^k(\mathbb{L}^{-1}T^n)^q,$$
$$1-\mathbb{L}^{-((\nu+q)p_1'+q_1')}T^{N_1'}=1-(\mathbb{L}^{-\nu}T^k)^{p_1'}(\mathbb{L}^{-1}T^n)^{q_1'+p_1'q}.$$
Then
$$\zeta _3(\Ic,\omega)(T)=(\mathbb{L}-1)^2 \Phi_3(x_1,x_2)\vert_{x_1=\mathbb{L}^{-\nu}T^k,x_2=\mathbb{L}^{-1}T^n},$$
where
$$\Phi_3(x_1,x_2)=\frac{D_0(x_1,x_2)}{(1-x_1x_2^q)(1-x_1^{p_1'}x_2^{q_1'+p_1'q})}$$
and
$$D_0(x_1,x_2)=\sum_{(i,j)\in \mathcal{P}_{\Delta_0'} }(x_1)^i (x_2)^{j+iq}.$$
The rational function $\Phi_1(x_1,x_2)$ is the generating function of the cone generated by $(1,q)$,
the rational function $\Phi_2(x_1,x_2)$ is the generating function of the cone generated by $(1,q)$ and $(p_{m-1},q_{m-1})$,
and the rational function
$\Phi_3(x_1,x_2)$ is the generating function of the cone generated by $(1,q),(p_1',q_1'+qp_1')$.
The map $(i,j)\to (i'=i,j'=j+qi)$ being a linear automorphism of the plane, we see that the configuration is as in Figure 6.

   \begin{figure}[ht]
 \begin{center}
\includegraphics{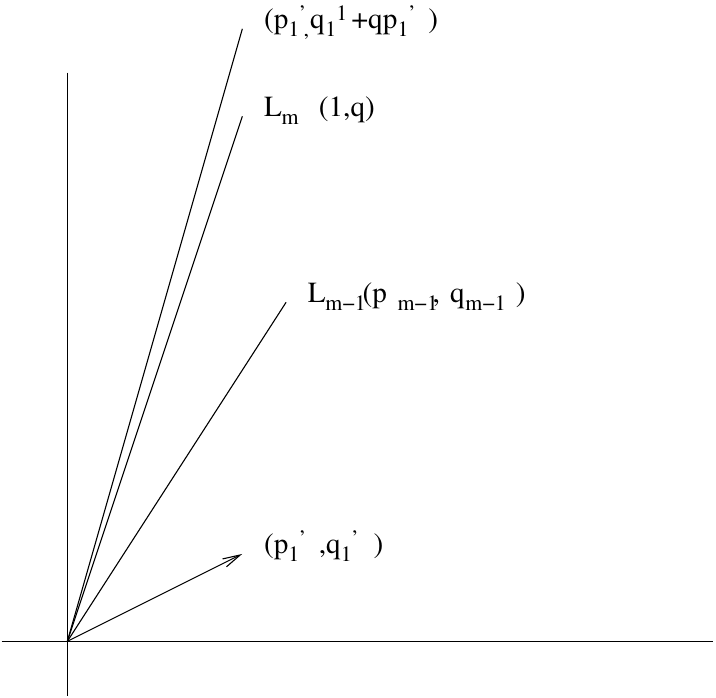}
 \caption{}
 \end{center}
\end{figure}

\noindent
 Hence $$\Phi_1(x_1,x_2)+\Phi_2(x_1,x_2)+\Phi_3(x_1,x_2)=\Phi_4(x_1,x_2),$$
where $\Phi_4(x_1,x_2)$ is the generating function of the cone
$\Delta$ generated by $(p_1',q_1'+qp_1')$ and $(p_{m-1},q_{m-1})$.

Then finally the sum of the four contributions is equal to
$$(\mathbb{L}-1)^2 \Phi_4(x_1,x_2)\vert_{x_1=\mathbb{L}^{-\nu}T^k,x_2=\mathbb{L}^{-1}T^n},$$
where
$$ \Phi_4(x_1,x_2)=\frac{\sum_{(i,j)\in \mathcal{P}_{\Delta}}x_1^ix_2^j}{(1-x_1^{p_1'}x_2^{q_1'+qp_1'})(1-x_1^{p_{m-1}}x_2^{q_m-1})}.$$
\end{proof}

\begin{remark}
We comment on two special cases above.

(1) When $S$ is the only face of the Newton polygon, the proof is
just easier.

(2) It is possible that one of the two factors in the denominator of
$\Phi_4$ is, after substitution, equal to
$1-\mathbb{L}^{-(\nu+q)}T^N$. In the statement of the proposition we
mean that the factor coming from $S$ cancels.

\end{remark}

\medskip
  \section{Newton tree associated with an ideal}

\medskip
In \cite{CV} we collected the information of the Newton algorithm of
an ideal in its Newton tree. It reflects the \lq tree shape\rq\ of
the algorithm, and keeps the information on the successive Newton
polygons. We recall briefly the main ideas.


 \smallskip
  \subsection{Graph associated with a Newton diagram.}

  A graph associated with a Newton diagram is a vertical linear graph with vertices, edges connecting vertices and two arrows at the top and the bottom.

  If the Newton polygon is empty, that is,
  $\Dg =(N,M)+\br_+^2$,
  the graph is in Figure 7. It has one edge connecting two arrows decorated by $N$ and $M$ at the top and the bottom, respectively.

   \begin{figure}[ht]
 \begin{center}
\includegraphics{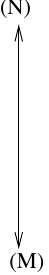}
 \caption{}
 \end{center}
\end{figure}

If the Newton polygon is $\cup_{1\leq i\leq m}S_i$, the graph has $m$ vertices $v_1,\cdots, v_m$ representing the faces $S_1, \cdots, S_m$. They are connected by edges when the faces intersect.
We add one edge at $v_1$ and at $v_m$ ended by an arrow.

  \begin{figure}[ht]
 \begin{center}
\includegraphics{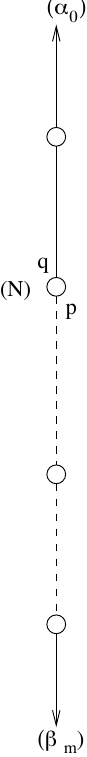}
 \caption{}
 \end{center}
\end{figure}

We decorate the vertices and the extremities of the edges near the
vertices using the following rule. Let $v$ be a vertex and $S$ be
the corresponding face whose supporting line has equation
$p\ag+q\bg=N$, where $(p,q)\in (\bn^*)^2$ and $\gcd (p,q)=1$.  We
decorate the vertex by $(N)$. Further we decorate the extremity of
the edge above the vertex with $q$ and the extremity of the edge
under the vertex by $p$; we say that the decorations near $v$ are
$(q,p)$. The arrows represent the  non-compact faces with supporting
lines $\{x=\ag_0\}$ and $\{y=\bg_m\}$; they are decorated with
$(\ag_0)$ at the top and $(\bg_m)$ at the bottom. See Figure 8.

 \subsection{Newton tree of an ideal}

 We build the Newton tree of $\Ic$ by induction on the depth.
 If its depth is zero, the ideal is generated by a \lq monomial\rq\ $x^N(y+h(x))^{M}$; we define its Newton tree to be the graph as in Figure 7.
Let now $\Ic$ be an ideal of depth $d(\Ic)$ greater than or equal to
one. We assume that we have constructed the Newton trees of ideals
of depths $d<d(\Ic)$.

On one hand we have the graph of the Newton polygon of the ideal $\Ic$. Consider a vertex $v$ on this
graph. It is associated with a face $S$ of the Newton polygon of $\Ic$ with equation
$p_S \ag+q_S\bg=N_S$ and
$$ \ti(\Ic,S)=x^{a_S}y^{b_S}F_{\Ic,S}(x^{q_S},y^{p_S})\big(k_1(x^{q_S},y^{p_S}),\cdots,k_s(x^{q_S},y^{p_S})\big)$$
with $\deg k_i=d_S\geq0$.  We decorate the vertex $v$ with the pair $(N_S,d_S)\in \bn ^2$.

Now we apply the Newton maps $\sigma=\sg_{(p_S,q_S,\mu_i)}$ for each
root $\mu_i$ of the face polynomial. (If the face is dicritical we
already know that the maps $\sg_{(p_S,q_S,\mu)}$ for $\mu$ generic
give a monomial ideal of the form $(x^{N_S})$ and we don't need to
perform those Newton maps.) The transformed ideal $\Ic _{\sg}$ has
depth less than $d(\Ic)$. Then from the induction hypothesis we can
construct the Newton tree of $\Ic _{\sg}$. It has a top arrow
decorated with $N_S$. We delete this arrow and glue the edge on the
vertex $v$. The edge which is glued on the vertex $v$ is a
horizontal edge. Horizontal edges join vertices corresponding to
different Newton polygons and vertical edges join vertices
corresponding to the same Newton polygon. Note that the \lq
width\rq\ of the Newton tree of $\Ic$ is precisely its depth
$d(\Ic)$.

We explain now how we decorate  the Newton tree. Let $v$ be a vertex
on the Newton tree of $\Ic$. If $v$ corresponds to a  face of the
Newton polygon of $\Ic$, we say that $v$ has {\it no preceding
vertex}  and we define $\cls (v)=\{v\}$.  If $v$ does not correspond
to a face of the Newton polygon of $\Ic$, it corresponds to a face
of the Newton polygon of some $\Ic _{\Sg}$. The Newton tree of $\Ic
_{\Sg}$ has been glued on a vertex $v_1$ which is called the {\it
preceding vertex} of $v$.  We note that the path between one vertex
and its preceding vertex contains exactly one horizontal edge but
may contain some vertical edges, for example as in Figure 9.

    \begin{figure}[ht]
 \begin{center}
\includegraphics{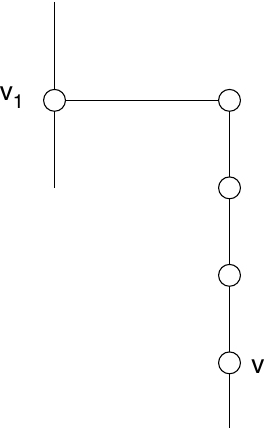}
 \caption{}
 \end{center}
\end{figure}

If $v_1$ does not correspond to  a face of the polygon of $\Ic$, we
can consider its preceding vertex, and so on. We define
$\cls(v)=\{v_i, \cdots, v_2,v_1,v\}$ where $v_j, 2\leq j\leq i,$ is
the preceding vertex of $v_{j-1}$, and $v_i$ corresponds to a face
of the Newton polygon of $\Ic$. The final Newton tree is decorated
in the following way.  Let $v$ be a vertex on the Newton tree of
$\Ic$. If $\cls(v)=\{v\}$, the decorations near $v$ are not changed.
If  $\cls(v)=\{v_i, \cdots, v_2,v_1,v\}$ and if the edge decorations near
$v$ on the Newton tree where $\cls(v)=\{v_{i-1}, \cdots,
v_2,v_1,v\}$ are $(m,p)$, then after the gluing on $v_i$ they become
$(m+p_iq_ip_{i-1}^2\cdots p_1^2p,p)$. The decorations of the vertices and of the arrows
are not changed. Usually we do not write the decoration of arrows decorated with $(1)$. We refer to \cite{CV} for motivations and
applications.

The vertices decorated with $(N,d)$ with $d>0$ (corresponding to
dicritical faces) are called {\it dicritical vertices}.

In $\S5$ we will use the notion of {\em edge determinant} of an edge $e$; this is the product of the edge decorations on $e$ minus the product of the adjacent edge decorations to $e$.

\bigskip
\noindent {\bf Example 1 {\rm (continued)}.} In Figure 10 we draw
the graphs associated with the occurring Newton diagrams, and the
resulting Newton tree.

  \begin{figure}[ht]
 \begin{center}
\includegraphics{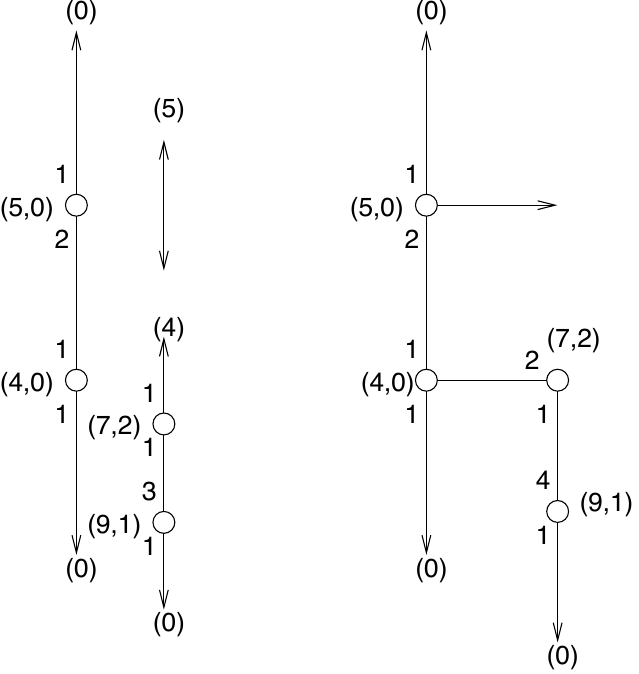}
\caption{}
 \end{center}
 \end{figure}

 \subsection{Minimal Newton trees}

The construction of the Newton trees shows that each vertex of the
Newton tree of an ideal $\Ic$ corresponds to a face of a Newton
polygon appearing in the Newton algorithm.

If the decorations near a vertex are $(q,p)$ and the decoration of
the vertex is $N$, then there is a possible contribution of the
vertex to the denominator of the motivic zeta function, namely
$(1-\mathbb{L}^{-(\nu p+q)}T^N)$. However, Proposition
\ref{no-contribution} shows that certain vertices do not contribute.
Those vertices correspond to faces of the Newton polygon which are
not dicritical faces, have equation $\ag+q\bg=N$ or $p\ag+\bg=N$,
hit respectively the $x$-axis or the $y$-axis, and have a face
polynomial with only one root. Such a vertex is the top vertex (on
the first vertical line) if its nearby decorations are $(1,p)$, it
is not dicritical and has valency $3$, and the top arrow is
decorated with $(0)$. It is a vertex at the bottom of the tree if
its nearby decorations are $(q,1)$, it is not dicritical and has
valency 3, and the bottom arrow is decorated with $(0)$. See Figure
11.
 Example 1 is an example with such vertices.

   \begin{figure}[ht]
 \begin{center}
\includegraphics{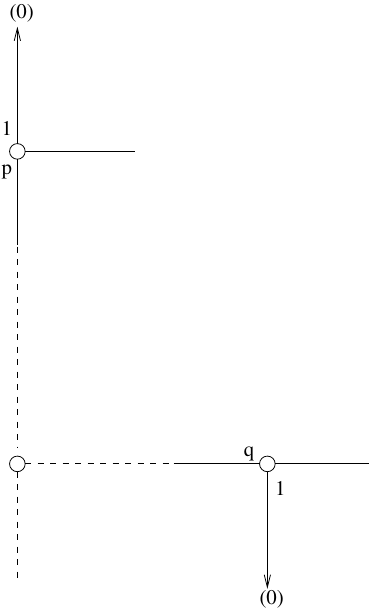}
\caption{}
 \end{center}
 \end{figure}

We will construct Newton trees where such vertices do not appear.
For that purpose we first introduce certain \lq manipulations\rq\ on
trees. These will turn out to correspond to coordinate changes. The
discussion will lead to a notion of minimal Newton tree, with a
geometric interpretation developed in $\S5$.

\smallskip
Take a vertex $v$ as above. Assume first that its nearby decorations
are $(q,1)$ such that either $q\neq 1$, or  $q=1$ and $v$ is not
connected to two arrows decorated with $(0)$. The edge $e$ decorated
with $q$ is connected to a vertex  or an arrow, denoted by $v_1$.
The edge $e_v$ decorated with $1$ is connected to an arrow decorated
with $(0)$ and the horizontal edge $e_h$ is connected to a vertex or
an arrow denoted by $v_2$. We delete  the vertex $v$ and the edge
$e_v$ with the  arrow at the other end. We connect $v_1$ and $v_2$
by an edge with the same orientation as $e$ and we decorate with the
decoration on $e$ near $v_1$ and with the decoration on $e_h$ near
$v_2$. We call this operation {\it erasing the vertex $v$}. See
Figure 12.

    \begin{figure}[ht]
 \begin{center}
\includegraphics{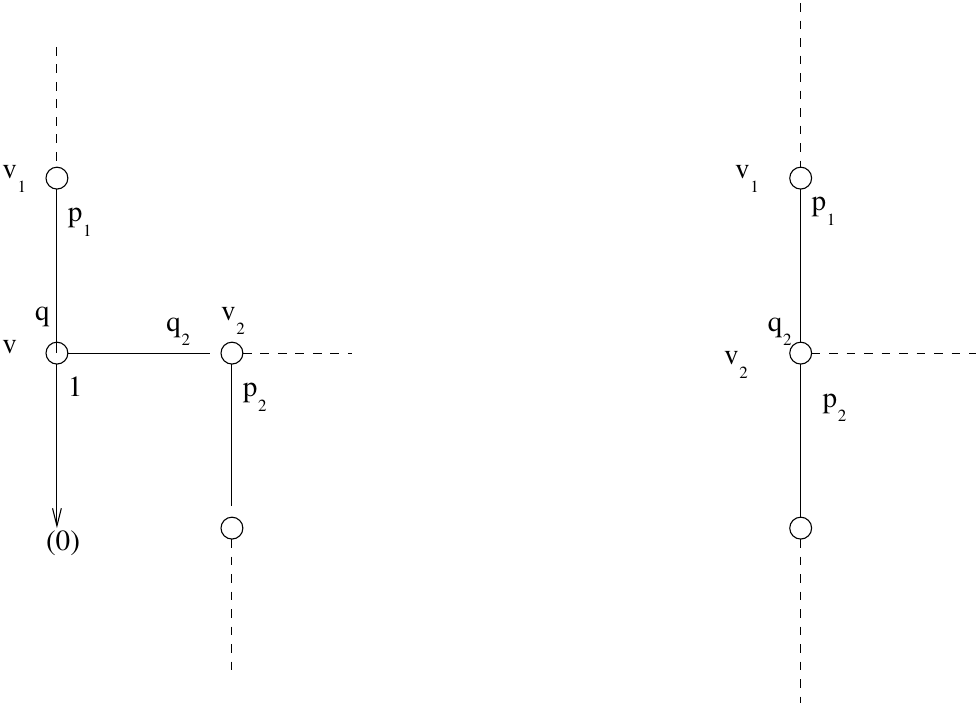}
\caption{}
 \end{center}
 \end{figure}

We act symmetrically for the first top vertex if it is decorated
with $(1,p)$ such that either $p\neq 1$, or  $p=1$ and $v$ is not
connected to two arrows decorated with $(0)$. Note that in this case
we change some decorations $(q_2,p_2)$ in $(p_2,q_2)$. See Figure
13.

   \begin{figure}[ht]
 \begin{center}
\includegraphics{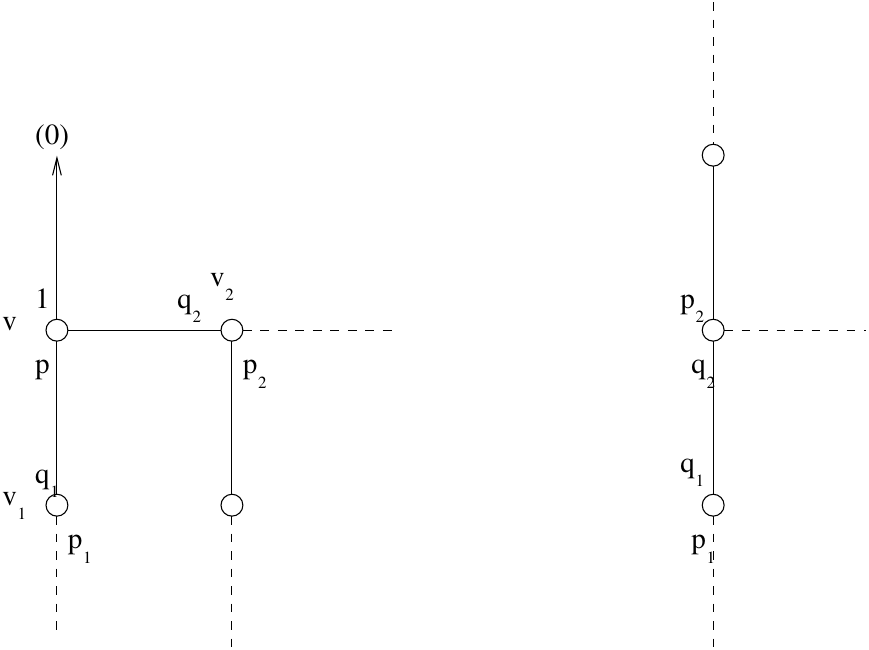}
\caption{}
 \end{center}
 \end{figure}

Let finally $v$ be decorated with $(1,1)$ and connected to two
arrows decorated with $(0)$.  The vertex $v$ is connected to a
vertex $v_2$ decorated with $(q_2,p_2)$ with $q_2>1$. We erase the
vertex $v$ as well as the edge with the bottom arrow. The top vertex
of the new Newton tree is $v_2$ decorated with $(q_2,p_2)$. See
Figure 14.

\smallskip
If above the vertex $v_2$ is a bottom vertex, with bottom arrow
decorated by $(0)$, if it has valence $3$ and is not dicritical, and
if $p_2=1$, then  we are in the previous case and we can also erase
$v_2$. We can continue this way.

    \begin{figure}[ht]
 \begin{center}
\includegraphics{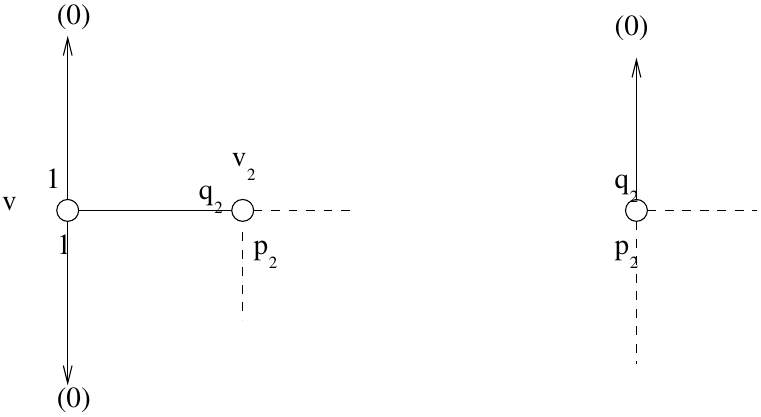}
\caption{}
 \end{center}
 \end{figure}

\begin{definition}
 A Newton tree is minimal if the only vertices connected to an arrow decorated with $(0)$ by an edge decorated with $1$ are dicritical vertices.
 \end{definition}

In order to obtain a minimal Newton tree, starting from any Newton
tree, erasing vertices (if applicable) is a good start. To finalize
the procedure we introduce a more general operation on
 Newton trees that
we call {\it exchange of vertical edge}.

 Let $\mathcal{T}$ be a
Newton tree. Assume we have a vertex $v$ decorated with $(q,1)$. Let
$e_1$ be the vertical edge arising from $v$ decorated with $p=1$ and
let $v_1$ be the vertex or the arrow at the other extremity of
$e_1$. Choose a horizontal edge arising from $v$, denote it by $e_2$
and  its other extremity by $v_2$. Let $\mathcal{T}'$ be the tree
obtained the following way. We cut $e_1$ and $e_2$ near the vertex
$v$. Then we obtain three subtrees: the part $\mathcal{T}_v$ of
$\mathcal{T}$ which contains $v$, the subtree $\mathcal{T}_1$ which
contains $e_1$ and the subtree $\mathcal{T}_2$ which contains $e_2$.
We stick $\mathcal{T}_1$ on $v$ by $e_1$ as a horizontal edge, not
changing the rest of $\mathcal{T}_1$, and we stick $\mathcal{T}_2$
on $v$ by $e_2$ as a vertical edge, not changing the rest of
$\mathcal{T}_2$. See Figure 15.

    \begin{figure}[ht]
 \begin{center}
\includegraphics{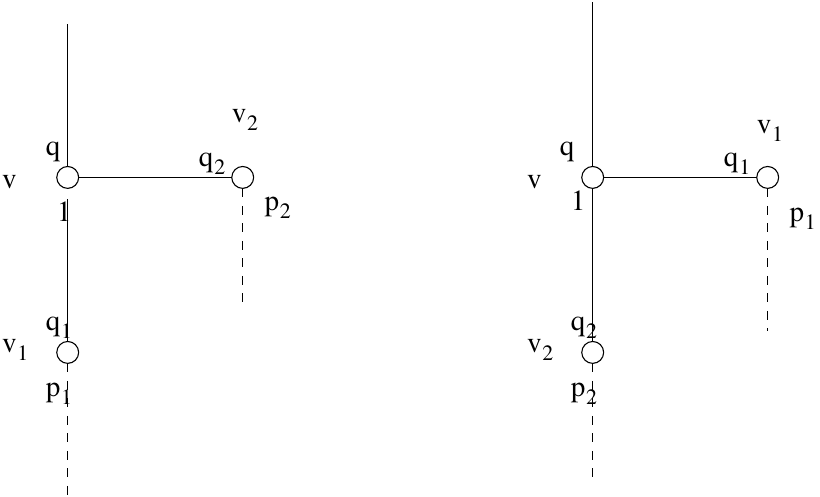}
\caption{}
 \end{center}
 \end{figure}

 \begin{remark}
Erasing a vertex $v$ is a particular case of exchanging vertical
edges. In fact, let $v$ be a vertex  decorated with $(q,1)$. We
assume that $v$ is not a dicritical vertex. Let $e_1$ be the
vertical edge decorated with $1$. We assume that $v_1$ is an arrow
decorated with $(0)$. Assume there is only one horizontal edge
$e_2$. Then we are in the case where we can erase $v$. We exchange
the edges $e_1$ and $e_2$. We can erase now the horizontal edge
because the arrow is decorated  with $(0)$ (In a Newton tree the horizontal edges are ending with vertices or with arrows decorated with strictly positive numbers.)
 Then we are
left with a vertical edge with a vertex of valency $2$ on it that we
can erase. Exchange of vertical edge can also be used to erase a
vertical edge decorated with $1$ and ending with an arrow decorated
with $(0)$, see Figure 16.
 \end{remark}

    \begin{figure}[ht]
 \begin{center}
\includegraphics{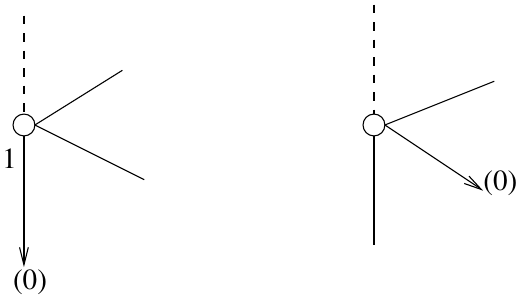}
\caption{}
 \end{center}
 \end{figure}

As long as there exists a non-dicritical vertex that is connected to
an arrow decorated with $(0)$ by an edge decorated with $1$, we can
clearly perform an exchange of vertical edge in order to erase it.
Hence we showed the following.

 \begin{prop}\label{towardsminimal}
By exchanging vertical edges one can obtain a minimal Newton tree,
starting from any Newton tree.
 \end{prop}

\noindent
 {\bf Example 1 {\rm (continued)}.}
A minimal Newton tree for the tree of Figure 10 is given in Figure
17.

    \begin{figure}[ht]
 \begin{center}
\includegraphics{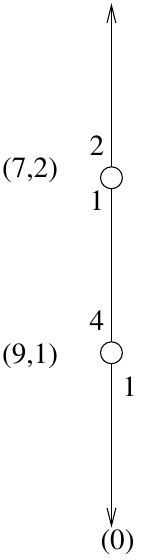}
\caption{}
 \end{center}
 \end{figure}

\begin{remark}
A minimal Newton tree of the generic curve of an ideal may not
coincide with a minimal Newton tree of the ideal. A dicritical
vertex of degree $1$ may be erased in the minimal Newton tree of the
generic curve but is not to be erased in the minimal Newton tree of
the ideal.
 \end{remark}

Now we show that minimal Newton trees can be obtained as Newton
trees in some coordinates.

 \begin{prop}\label{exchange}
Let $\Ic \subset \bc[[x,y]]$ be a non-trivial ideal with Newton tree
$\mathcal{T}$. Assume there is on the first vertical line of
$\mathcal{T}$ a vertex $v$ decorated with $(q,1)$. Choose a
horizontal edge to make an exchange of vertical edge. Let
$\mathcal{T}'$ be the resulting tree. Then there exists a change of
coordinates  $\phi:(x,y)\mapsto (x',y')$ such that $\mathcal{T}'$ is
the Newton tree of the ideal $\Ic'=\phi(\Ic)\subset \bc[[x',y']]$.

A similar result holds for a vertex on the first vertical line
decorated with $(1,p)$.
 \end{prop}

 \begin{proof}
Let $S$ be the face of the Newton polygon corresponding to $v$.
Write
$$
 \ti(\Ic,S)=x^{a_S}y^{b_S}F_{\Ic,S}(x^{q_S},y^{})\big(k_1(x^{q_S},y^{}),\cdots,k_s(x^{q_S},y^{})\big)
$$
 with
 $$F_{\Ic,S}(x^{q_S},y)=\prod_{1\leq i\leq n}(y-\mu_ix^{q_S})^{\nu_i}$$
 as in (\ref{initialideal1}) or (\ref{initialideal2}).

Consider first the case where the horizontal edge we have chosen
does not end with an arrow. Let $\mu$ be the root (with multiplicity
$\nu$) corresponding to the edge we have chosen. We perform the
change of variables
 $$x'=x,\ \ \ \ y'=y-\mu x^{q_S}.$$
The faces above $S$ do not change. The face $S$ transforms into a
face $S'$ with the same $(q_S,1)$. Its associated initial ideal is
$$
 \ti(\Ic,S')=(x')^{a_S}(y')^{\nu}F_{\Ic,S'}((x')^{q_S},y')\big(k'_1((x')^{q_S},y'),\cdots,k'_s((x')^{q_S},y')\big)
$$
with
$$F_{\Ic,S'}\big((x')^{q_S},y'\big)=\big(y'+\mu (x')^{q_S}\big)^{b_S}\prod_{1\leq i\leq n,\mu_i \neq \mu}\big(y'-(\mu_i-\mu)(x')^{q_S}\big)^{\nu_i}.$$
For $\mu_i \neq \mu$ we perform in the $(x,y)$-coordinates the
Newton map $\sigma_{(1,q_S,\mu_i)}:$
 $$x=x_1, \qquad  y=x_1^{q_S}(y_1+\mu_i),$$
and in the $(x',y')$-coordinates the map $\sigma_{(1,q_S,\mu_i -
\mu)}:$
 $$x'=x'_1, \qquad y'=(x'_1)^{q_S}(y'_1+\mu_i-\mu).$$
A simple computation yields $x_1=x'_1$ and $y_1=y'_1$.

Now we replace the Newton map $\sg_{(1,q_S,\mu)}$ by  the change of
variables
 $$x'=x, \qquad y'=y-\mu x^{q_S},$$
 yielding
 $$x'=x_1, \qquad y'=x_1^{q_S}y_1.$$
Then $x_1^{\ag}y_1^{\bg}= (x')^{\ag-q_S\bg}(y')^{\bg}.$ Let
$\ag'=\ag-q_S \bg$, $\bg'=\bg$. This is a linear automorphism of
the plane.  Any face of the original second Newton polygon with
equation $p\ag+r\bg=N$ corresponds to a face on the new first Newton
polygon with equation
 $p\ag'+(r+pq_S)\bg'=N$, hence lying under $S'$.
Analogously, with the Newton
 map $\sigma_{(1,q_S, - \mu)}:$
 $$x'=x'_1, \qquad y'= (x'_1)^{q_S}(y'_1-\mu)$$
one verifies that $x^{\ag}y^{\bg}= (x'_1)^{\ag+q_S\bg}(y'_1)^{\bg}$
and that a face under $S$ of the original first polygon corresponds
to a face of the new second polygon. Consequently $\mathcal{T}'$ is
indeed the Newton tree of  $\Ic'$.

If the horizontal edge we have chosen ends with an arrow, we have
$\Ic=(y+h(x))^{\nu}\Ic_1$, with $h(x)\in x\bc[[x]]$, $h(x)\neq 0$
and $\nu >0$. In that case, we perform the change of variables
$x'=x,y'=y+h(x)$.
 \end{proof}

\bigskip

\noindent
 {\bf Example 1 {\rm (continued)}.}
To obtain the minimal Newton tree in Figure 17, we first perform the
change of coordinates
$$x'=x, \qquad y'=y+x .$$
This leads to
$\Ic'=((y'-x')^2-3x')\big(y'(y'-x')^4,(y')^{3}+(x')^{8}\big)$. We
can write $(y'-x')^2-3x'$ as the product of a unit and a series of
the form $x'+h(y')$ with $h(y')\in y'\bc[[y']]$. After a second
change of coordinates
$$x"=x'+h(y'),  \qquad y"=y'$$
We obtain an ideal in $\bc[[x",y"]]$ with the tree in Figure 17 as
Newton tree.

\bigskip

\begin{remark}
Minimal Newton trees are not unique. For example, in Example 1 one
can alternatively perform first the coordinate change
$$x'=-3x+y^2, \qquad y'=y .$$
After a second coordinate change
$$x"=x',  \qquad y"=y'-\frac 13 x'$$
the resulting ideal has the Newton tree given in Figure 18.

    \begin{figure}[ht]
 \begin{center}
\includegraphics{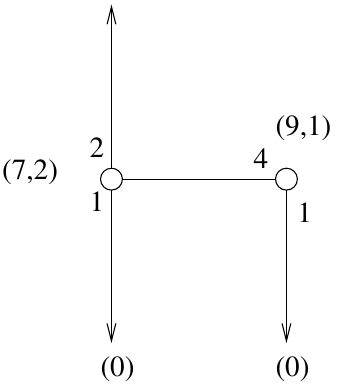}
\caption{}
 \end{center}
 \end{figure}

\end{remark}

\medskip
Next we introduce the notion of (very) good coordinates; the point is that a Newton tree is minimal if and only if the Newton algorithm is performed in very good coordinates at each step.

 \begin{definition}
Let $\Ic$ be a non-trivial ideal in $\bc[[x,y]]$ and $\nc(\Ic)=\cup
_{1\leq i\leq m}S_i$ its Newton polygon. We say that $\Ic$ is in
{\em good coordinates} if $\ti (\Ic,S_m)$ is not a principal ideal
generated by a polynomial of the form $x^k(y-\mu x^q)^l$ with $k\geq
0$, $l\geq 1$ and $\mu \in \bc ^*$.
 \end{definition}

\begin{lemma}
 Let $\Ic$ be a non-trivial ideal in $\bc[[x,y]]$. There exists a system of coordinates which is good for $\Ic$.
 \end{lemma}

 \begin{proof}
Assume that the ideal $\Ic$ is not in good coordinates. Then there
exist $k\geq 0$, $l\geq 1$ and $\mu \in \bc ^*$
 such that
 $$\ti (\Ic,S_m)=(x^k(y-\mu x^q)^l) .$$
Consider the automorphism of $\bc^2$ defined by $x = x'$, $y =
y'+\mu (x')^q$. This automorphism does not change the faces $S_1,
\cdots,S_{m-1}$ of $\nc(\Ic)$ but makes the face $S_m$ disappear. If
this new system of coordinates is not good for $\Ic$, then we
consider a new automorphism and so on. The process is finite if the
ideal has $\{0\}$ as support. However, if the support of $\Ic$
contains a curve, that process can be infinite. In that case we use
a coordinate change of the form $x=x'$, $y=y'+h(x')$ with $h(x')\in
x'\bc[[x']]$. The details for this argument are the same as in the
case of a principal ideal as explained in e.g. \cite[Lemma 1.3]{LO}
or \cite[Theorem 6.40]{KSC}.
 \end{proof}

As an example of the first case, the ideal
$\Ic_1=(y^4,x^2(x^2+y)+x^3y^2)$ is not in good coordinates. We
obtain good coordinates after the change $x=x'$, $y=y'-(x')^2$. For
the second case, consider $\Ic_2 =(y+x^2+xy)$. Here the system of
good coordinates is obtained by $x = x'$, $y = y' -
(x')^2(1+x')^{-1}$.

 \begin{remark}
Note that, in the procedure above, we do not eliminate dicritical
faces.
 \end{remark}

 \begin{definition}
 Let $\Ic$ be a non-trivial ideal in $\bc[[x,y]]$ and $\nc(\Ic)=\cup _{1\leq i\leq m}S_i$ its Newton polygon. We say that
$ \Ic$ is in {\em very good coordinates} if the following conditions are satisfied.
\begin{itemize}
\item For $m>1$: $\Ic$ is in good coordinates and $\ti (\Ic,S_1)$ is not a principal ideal generated by a polynomial of the form $y^k(y^p-\mu x)^l$
with $k\geq 0$, $l\geq 1$ and $\mu \in \bc ^*$.

\item For $m=1$: $\ti (\Ic,S_1)$ is not a principal ideal generated by a polynomial of the form
$x^k(y-\mu x^q)^l$ or $y^k(y^p-\mu x)^l$ as before,  or $(y-\mu_1
x)^{l_1}(y-\mu_2 x)^{l_2}$ with $l_1,l_2\geq 1$ and $\mu_1,\mu_2 \in
\bc^*$.
\end{itemize}
 \end{definition}

\noindent
 One verifies analogously that there exists a system of
coordinates which is very good for $\Ic$. (For the last case one
performs the coordinate change $y-\mu_1 x = x'$, $y-\mu_2 x = y'$.)

\bigskip
For example, the ideal $\Ic=(x^2-y^2,x^3)$ is not in very good
coordinates. We obtain very good coordinates after the change
$x=(x'+y')/2$, $y=(y'-x')/2$.

\begin{remark}
If $\Ic=(x^N)\mathcal{J}$ with $N\geq 1$, then good coordinates for $\Ic$ implies very good coordinates for $\Ic$.
\end{remark}


The previous discussion shows the following.

\begin{remark}
The Newton tree is minimal if and only if the Newton algorithm is performed in very good coordinates at each step.
\end{remark}

\bigskip
\begin{remark}
The notion of depth of an ideal from \cite{CV}, see Definition
\ref{def-depth}, depends heavily on the chosen coordinates. In
\cite{CV} we alluded to a possible more intrinsic notion. We can now
define the {\em geometric depth} of an ideal in $\bc [[x,y]]$ as the
minimum of the widths of the minimal Newton trees of $\Ic$.

For example the ideal in Example 1 has depth 2 but geometric depth 1.
\end{remark}

\medskip
 \section{Geometric interpretation}

\medskip
In the case of a principal ideal $\Ic=(f)$ of $\bc [[x,y]]$ it is
shown in \cite[Theorem 3.6]{CNL} that the Newton tree of $\Ic$ is
more or less equal to the dual graph of the minimal embedded
resolution of $f$, where vertices of valence $2$ are deleted. See
also \cite{LO} (especially Lemma 2.12 and Theorem 3.1.2(i)). These
proofs immediately generalize to arbitrary ideals, yielding the
results in the theorem below.

We first recall the notion of (minimal) log principalization of an ideal (in our context).

\medskip
Let $\Ic$ be a non-trivial ideal in $\bc [[x,y]]$. A {\em log
principalization} (or {\em monomialization}) of $\Ic$ is a proper
birational morphism $h: X \to \spec \bc [[x,y]]$ from a smooth
variety $X$ such that the ideal sheaf $\Ic \mathcal{O}_{\bar X}$ is
invertible and moreover its associated divisor of zeroes is a
(strict) normal crossings divisor. The {\em minimal log
principalization} of $\Ic$ is the unique such $h$, up to
isomorphism, through which all other ones factorize.

\begin{theorem}\label{principalization}
Let $\Ic$ be a non-trivial ideal in $\bc [[x,y]]$ and $T$ a minimal
Newton tree of $\Ic$. There exists a toroidal surface $\bar X$ and a
proper birational morphism $\bar \pi:\bar X \to \spec \bc [[x,y]]$
with the following properties. Let $E_j,j\in J,$ denote the
irreducible components of the exceptional locus of $\bar \pi$ and
$C_i,i\in I,$ the strict transform of the one-dimensional
irreducible components of the support of $\Ic$ (if any).

(1) There is a natural bijection between the $E_j$ and the vertices
of $T$, and between the $C_i$ and the arrowheads of $T$ that are not
decorated by $(0)$.

(2) The ideal sheaf $\Ic \mathcal{O}_{\bar X}$ is invertible, with associated
normal crossings divisor (in orbifold sense) $\bar D=\sum_{i\in I}
N_i C_i + \sum_{j\in J} N_j E_j$, where the coefficients $N_j$ and
$N_i$ are the decorations of the corresponding vertices and arrows,
respectively, of $T$.

(3) The singularities of $\bar X$ are located on the intersections
of the components $E_j$. They are cyclic quotient singularities of
order equal to the edge determinant of the corresponding edge in
$T$.

The composition of the resolution map $h:X\to \bar X$ of these
singularities and the map $\bar \pi$ yields the minimal log
principalization map $\pi : X \to \spec \bc [[x,y]]$ of the ideal $\Ic$.
\end{theorem}

Morally the tree $T$ is the dual graph of the divisor $\bar D$ in
$(2)$. More precisely, one obtains this dual graph by erasing in $T$
the arrowheads decorated with $(0)$  and all edges starting from
those arrowheads and decorated by $1$ on the other side.

What was still missing is the meaning of the edge decorations of $T$
in the dual graph of $\bar D$. The next proposition says that these
edge decorations have the \lq correct\rq\ interpretation as classical edge
decorations in a dual graph.

\begin{prop}\label{edgedecorations} We use the notation of Theorem
\ref{principalization}.  Denote $D=h^* \bar D$. The decoration on an
edge $e$ near a vertex $v$ of the Newton tree $T$ is the absolute
value of the determinant of the intersection matrix of those
components of $D$ (in the surface $X$), corresponding to the part of
the tree away from $v$ in the direction of $e$.
\end{prop}

\begin{proof}
We first fix some terminology. We will often talk about the
determinant of minus the intersection matrix (which is the same as
the absolute value of the determinant of the intersection matrix) of
a configuration $C$ of (irreducible) curves on a smooth surface. For
sake of brevity we denote $\det(C)$ for this determinant.

Recall that $\det(\cdot)$ is invariant under blowing up: when $C$ is
a configuration of curves on a smooth surface $S$ and $\tilde S \to
S$ is a point blow-up with exceptional curve $E$, then, denoting by
$\tilde C$ the configuration in $\tilde S$ of strict transforms of
curves in $C$, we have $\det(C)=\det(\tilde C \cup \{E\})$.

Also, with the configuration $C$ of curves associated to a part
$\Delta$ of the Newton tree $T$ we mean the total transform in $X$
of the configuration of curves in $\bar X$, associated by the
bijection in Theorem \ref{principalization}(1) to the vertices in
$\Delta$. We then define $\det(\Delta):=\det(C)$. With this
terminology we must show:

\smallskip
\begin{center}
($*$)\hfill
\begin{minipage}{.94\linewidth}
{\em any decoration on an edge $e$ near a vertex $v$ of $T$ is equal
to $\det(\Delta)$, where $\Delta$ is the connected component of
$T\setminus \{v\}$ in the direction of $e$.}
\end{minipage}
\end{center}

\smallskip
\noindent
(i) Any (implicit) edge decoration $1$ on the left of a
horizontal edge (i.e., right to a vertex) satisfies ($*$).

\smallskip
Indeed, the associated part $\Delta$ corresponds in $X$ to the total
exceptional locus of a composition of blow-ups.

\smallskip
\noindent (ii) Any edge decoration on the first vertical line
satisfies ($*$).

\smallskip
Fix a vertex $v$ on the first vertical line of $T$ with edge
decorations $q$ and $p$ above and below $v$, respectively. Denote by
$Q_1$ and $P_1$ the parts of $T$ {\em on the first vertical line}
above and below $v$, respectively, and by $Q$ and $P$ the total
connected components of $T\setminus \{v\}$ in the direction of $q$
and $p$, respectively. E.g. \cite[Theorem 10.4.4]{CLS} immediately
yields $q=\det(Q_1)$ and $p=det(P_1)$. And because $\det(\cdot)$ is
invariant under blowing up, we have $\det(Q_1)=\det(Q)$ and
$\det(P_1)=\det(P)$. We conclude that $q=\det(Q)$ and $p=\det(P)$,
finishing case (ii).

\smallskip
\noindent (iii) Any edge decoration below a vertex on another
vertical line satisfies ($*$).

\smallskip
Because the edge decorations below vertices do not change during the
construction of the Newton tree $T$, we can use exactly the same
arguments as in the previous case.

\smallskip
\noindent (iv) Any edge decoration above (or left to) a vertex on
another vertical line satisfies ($*$).

\smallskip Here we use induction on the number of preceding vertices
of the given vertex. Fix a vertex $v$ with edge decorations $q$ and
$p$, and say $v_0$, with edge decorations $q_0$ and $p_0$, is its
preceding vertex (see Figure 19).

    \begin{figure}[ht]
 \begin{center}
\includegraphics{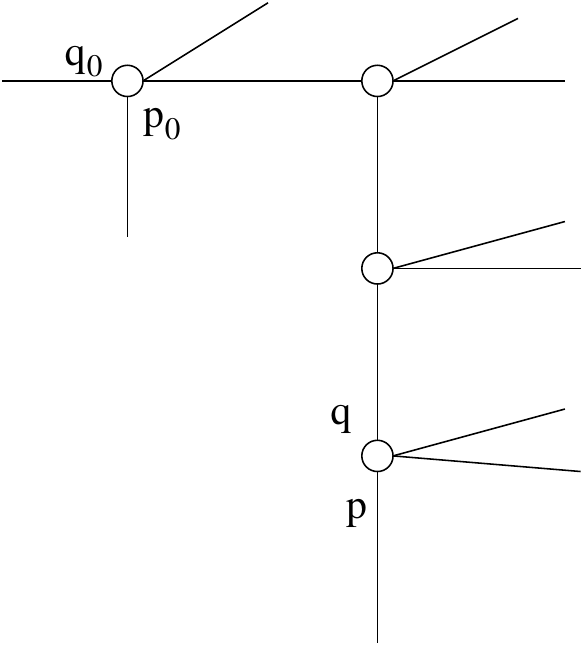}
\caption{}
 \end{center}
 \end{figure}

Let also $m$ be the original
decoration above $v$ when drawing the vertical line of $v$
associated to a Newton diagram. Then by \cite[Proposition 3.1]{CV}
we have
 \begin{equation}\label{equality1}
 q=p_0q_0p+m .
 \end{equation}
We introduce notation for the relevant parts of $T$. Let $Q$ and $P$
be the connected components of $T\setminus \{v\}$ in the direction
of $q$ and $p$, respectively, and similarly $Q_0$ and $P_0$. Let
also $\Delta$ denote the chain between $v$ and $v_0$, i.e., the part
above $v$ of the original vertical line of $v$ associated to a
Newton diagram. As in (ii), we have that
\begin{equation}\label{equality2}
m=\det(\Delta),\quad p=\det(P) \quad\text{and}\quad p_0=\det(P_0).
 \end{equation}
Our induction hypothesis says that
\begin{equation}\label{equality3}
q_0=\det(Q_0).
 \end{equation}
Combining (\ref{equality1}), (\ref{equality2}) and (\ref{equality3})
yields $q=\det(P_0)\det(Q_0)\det(P)+\det(\Delta)$. Finally by a
(well known) determinant computation this right hand side is equal
to $\det(Q)$.
\end{proof}

Under the bijective correspondence in \ref{principalization} between vertices of $T$ and exceptional components of $\bar \pi$, the dicritical vertices of $T$ have a conceptual geometric meaning. We fix some notation in order to describe this relation.

Let $b: B\to \spec \bc [[x,y]]$ denote the normalized blow-up of the ideal $\Ic$. Since the ideal sheaf $\Ic\mathcal{O}_{\bar X}$ is invertible, the map $\bar \pi :\bar X \to \spec \bc [[x,y]]$ factorizes through $b$, yielding a morphism $p:\bar X \to B$.

\begin{theorem} We use the notation of Theorem
\ref{principalization}.
Under the bijection (1), the dicritical vertices of the Newton tree $T$  correspond with the exceptional components of the normalized blow-up of the ideal $\Ic$, that is, with those exceptional components of $\bar \pi$ that are mapped onto a curve by $p$.
\end{theorem}

\begin{proof} By
\cite[Proposition 5.12]{CV} the dicritical vertices of $T$ correspond to the Rees valuations of $\Ic$. Next use \cite{SH}.
\end{proof}

\medskip
In fact the morphism $p:\bar X \to B$ is a natural and important \lq
partial resolution\rq\ of the divisor of $\Ic\mathcal{O}_{B}$ in $B$
in the framework of the Minimal Model Program; it is precisely the
relative log canonical model of this divisor in $B$. We refer to
\cite{KM} for definitions.

\medskip

\begin{prop}
The motivic zeta function of an ideal in $\bc [[x,y]]$ only depends
on a minimal Newton tree of the ideal. When performing the Newton
algorithm in very good coordinates, all candidate poles in the
formula of Theorem \ref{formulazetafunction} are true poles.
\end{prop}

\begin{proof}. The first claim is clear. The second one follows from
\cite{VV1} and the geometric interpretation of the Newton tree.
\end{proof}

\medskip
 \section{The log canonical threshold}

\medskip

The log canonical threshold is an important singularity invariant.
We refer to \cite{Mu} for various equivalent definitions and more
information. We only mention the description that we need,  for
ideals in $\bc [[x,y]]$,  using the notation of Theorem
\ref{principalization}.

 Recall that
$\pi : X \to \spec \bc [[x,y]]$ is the minimal log principalization
of $\Ic$. We denote the irreducible components of $D=\Ic
\mathcal{O}_{X}$ by $F_\ell, \ell \in L$; they consist of the
exceptional components of $h$ and (the strict transforms of) the
$E_j$ and $C_i$. Write $D=\sum_{\ell\in L} N_\ell F_\ell$ and $\pi^*
(dx\wedge dy) = \sum_{\ell\in L} (n_\ell -1)E_\ell$.

\begin{definition}
Let $\Ic$ be a non-trivial ideal in $\bc [[x,y]]$. The {\em log
canonical threshold} $lct(\Ic)$ of $\Ic$ is the minimum of the
numbers $\frac{n_\ell}{N_\ell}, \ell \in L$.
\end{definition}

\begin{remark}\label{lctVV}
The minimum above is in fact always obtained on components of $\bar
D$, that is, $lct(\Ic)$ is the minimum of the numbers
$\frac{n_\ell}{N_\ell}, \ell \in I\cup J$.  See e.g. \cite[Section
3]{VV1}.
\end{remark}

Under the correspondence in Theorem \ref{principalization} one can
also describe the $n_\ell$ in terms of the Newton tree $T$. Recall
that a component of $\bar D$ corresponds to  a vertex or arrowhead
in the tree, hence also to a face of some Newton diagram appearing
in the Newton algorithm of  $\Ic$.

 We will use the following
terminology: after a composition of Newton maps, we call the
pullback of $dx\wedge dy$ {\em the differential form at that stage
of the Newton algorithm}.

\begin{lemma}\label{differentialformintree}
Let $\Ic$ be a non-trivial ideal in $\bc [[x,y]]$. Let $S$ be a
face, with equation $p_S\ag+q_S\bg=N_S$, of some Newton diagram
appearing in the Newton algorithm of  $\Ic$. Let $F_S$ be the
component of $D$ corresponding to $S$. Then $n_S= p_S\nu+q_S$, where
$x^{\nu-1}dx \wedge dy$ is
 the differential form at that stage of the Newton algorithm (in local coordinates $x,y$).
\end{lemma}

\begin{proof}
This follows immediately from \cite[Proposition 3.9]{CNL}.
\end{proof}

\begin{prop}
The log canonical threshold of a non-trivial ideal $\Ic$ in $\bc
[[x,y]]$ is the minimum of the numbers $\frac{p_S\nu+q_S}{N_S}$,
where $S$ runs over all
 faces, with equation $p_S\ag+q_S\bg=N_S$,
 of all the Newton diagrams appearing in the Newton algorithm of $\Ic$,
 and where $x^{\nu-1}dx \wedge dy$ is
 the differential form at that stage of the Newton algorithm.
 \end{prop}

\begin{proof}
Combine Remark \ref{lctVV} and Lemma \ref{differentialformintree}.
 \end{proof}

\begin{theorem}\label{lct}
 Let $\Ic$ be a non-trivial ideal in $\bc[[x,y]]$. Then there exists a change of coordinates
 $\phi: (x,y)\mapsto (x',y')$ such that $lct(\Ic)=\frac{p+q}{N}$ where $S$, with equation $p\ag+q\bg=N$,
 is the face of the first Newton diagram of the ideal $\Ic'=\phi(\Ic)$ in $\bc[[x',y']]$, intersected by the line $\ag=\bg$.
 \end{theorem}

To prove the theorem, we need a series of lemmas. The first one is
 quite obvious.

\begin{lemma}
 Let $\mathcal{N}$ be a Newton diagram whose faces have equation $p_i\alpha +q_i \bg=N_i$ with
 $(p_i,q_i)=1$ for $i=0,\cdots, m+1$ and $p_0=1,q_0=0$, $p_{m+1}=0,q_{m+1}=1$. Let $\nu$ be a positive integer.
 Then $\min_{0\leq i \leq m+1} \frac{p_i\nu+q_i}{N_i} = \frac{1}{t}$, where $t$ is the ordinate of the point of intersection of the line $\ag=\nu\bg$ with the Newton diagram.
 \end{lemma}

\begin{cor}
 The minimum of the values $\frac{p\nu+q}{N}$ on the first Newton diagram is obtained by $1/t$, where $t$ is the ordinate of the point where the line $\ag=\nu \bg$ intersects the Newton diagram.
 \end{cor}

 \begin{definition}
Let $S$ be a face, with equation $p \ag+q \bg=N$, of a Newton
diagram appearing in the Newton algorithm, with differential
$w=x^{\nu-1}dx \wedge dy$ at that stage. The rational number
$\frac{n}{N}=\frac{p\nu+q}{N}$ is called the {\em invariant of $S$}.
 \end{definition}

We are looking for the smallest invariant appearing in the Newton
process.

  \begin{cor}\label{nm<N}
Let $S_0$ be a face of a Newton polygon where the invariant is
$\frac{n_0}{N_0}$. Let $\mu_0$ be a root of its face polynomial with
multiplicity $m_0$. Assume $m_0n_0<N_0$. Consider the Newton map
associated to $S_0$ and $\mu_0$. Then for all faces of the new
Newton diagram, the invariant $\frac{n}{N}$ satisfies
$\frac{n}{N}\geq \frac{n_0}{N_0}$.
  \end{cor}

  \begin{proof}
After applying the Newton map $\sigma_{(p,q,\mu_0)}$ associated to
$S_0$ and $\mu_0$, we have a Newton diagram with a non-compact face
$\ag=N_0$, and the highest point on the Newton polygon has
coordinates $(N_0,m_0)$. The hypothesis $m_0n_0<N_0$ says that the
line $\ag=n_0\bg$ intersects the face $\ag=N_0$ (see Figure 20).
Then the minimum of the values $\frac{pn_0+q}{N}$ on all the faces
of the new diagram is $\frac{n_0}{N_0}$. But the differential at
this stage of the Newton process is $w=x^{n_0-1}dx\wedge dy$.

     \begin{figure}[ht]
 \begin{center}
\includegraphics{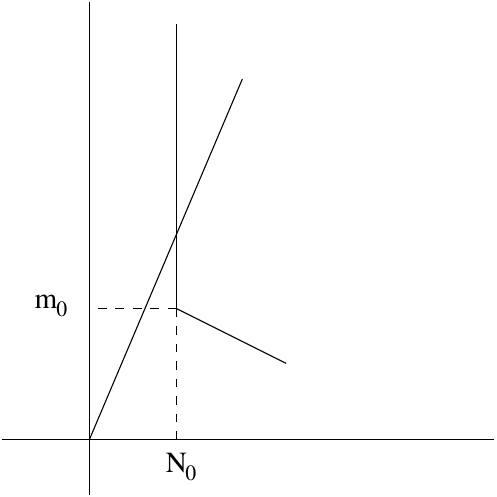}
\caption{}
 \end{center}
 \end{figure}

 \end{proof}

 \begin{lemma}\label{nm<Nbis}
Let $S_0$ be a face of a Newton polygon where the invariant is
$\frac{n_0}{N_0}$. Let $\mu_0$ be a root of its face polynomial with
multiplicity $m_0$. Consider the Newton map associated to $S_0$ and
$\mu_0$. Let $S$ be a face of the new Newton polygon with invariant
$\frac{n}{N}$, and let $\mu$ be a root of its face polynomial with
multiplicity $m$. If $n_0m_0<N_0$, then $nm<N$.
 \end{lemma}

\begin{proof}
Write the equation of $S$ as $p\ag+q\bg=N$, and let $m'$ be the
ordinate of the point where $S$ intersects the line $\alpha=N_0$
(Figure 21). We have $pN_0+qm'=N$ and
 $$m\leq m'\leq m_0<\frac{N_0}{n_0}.$$
 Then
 $$mn-N=m(pn_0+q)-N\leq pm_0n_0+m'q-N<pN_0+qm'-N=0.$$
\end{proof}

       \begin{figure}[ht]
 \begin{center}
\includegraphics{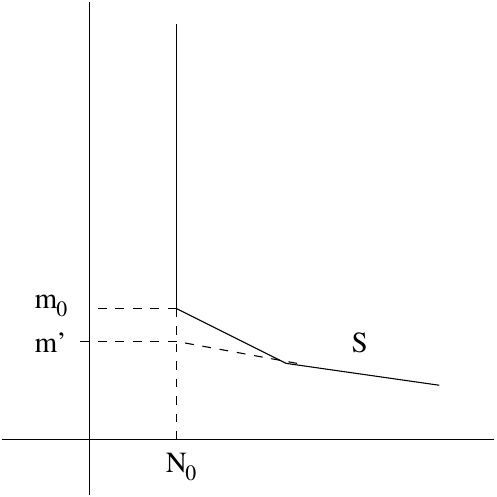}
\caption{}
 \end{center}
 \end{figure}

Next we study the case where $n_0m_0 \geq N_0$.

\begin{lemma}\label{nm>N}
Let  $S_0$ be a face of a Newton polygon with equation $p_0
\ag+q_0\bg=N_0$ and with invariant $\frac{p_0+q_0}{N_0}$. Let
$\mu_0$ be a root of its face polynomial with multiplicity $m_0$. If
$n_0m_0 \geq N_0$, then $p_0$ or $q_0$ is equal to $1$.
  \end{lemma}

\begin{proof}
We associate to the initial ideal with respect to $S_0$ the numbers
$a$, $b$ and $d$ as in (\ref{initialideal1}) or
(\ref{initialideal2}), and we denote by $m_\mu$ de multiplicity of
another root $\mu$ of the face polynomial of $S_0$. Then we have
  $$N_0=ap_0+bq_0+p_0q_0(m_0+\sum_{\mu \neq \mu_0}m_{\mu}+d).$$
Hence
  $$0\leq n_0 m_0 - N_0=(p_0+q_0)m_0-N_0 \leq (p_0+q_0)m_0-p_0q_0m_0.$$
And $p_0+q_0-p_0q_0 \geq 0$ implies that $p_0$ or $q_0$ is equal to
$1$ (since $p_0$ and $q_0$ are coprime).
  \end{proof}

\bigskip
Now we can prove Theorem \ref{lct}.

\begin{proof}
We fix a vertex (or arrow) $v$ in the Newton tree of $\Ic$ where the
minimum value of the invariant $\frac nN$ is attained, such that no
preceding vertex of $v$ has that minimum as invariant.

 If $v$ corresponds to
a face on the first Newton polygon, we know that the minimum is
given by the intersection with the line $\ag=\bg$. If it is not on
the first vertical line of the Newton tree, then we consider the
sequence of vertices which lead from the first vertical line to $v$.
 Let $v_0$ be the vertex of that sequence on
the first vertical line. We claim that $n_0 m_0 \geq N_0$. Indeed,
otherwise Lemma \ref{nm<Nbis} and Corollary \ref{nm<N} provide a
contradiction.

Hence by Lemma \ref{nm>N} we have that $p_0$ or $q_0$ is equal to
$1$. Say $p_0=1$. Let $v_1$ be the second vertex of the sequence. We
can perform the exchange of vertical edge between the edge decorated
by $p_0=1$ and the horizontal edge with extremities $v_0$ and $v_1$.
If $v_1=v$, then $v$ is now on the first Newton polygon and as its
invariant is the minimum, it is given by the intersection with the
line $\ag=\bg$. If $v_1\neq v$, we can conclude similarly that $n_1
m_1 \geq N_1$, apply Lemma \ref{nm>N} to $v_1$ and perform an
exchange of vertical edge. We continue until $v$ is on the first
vertical line. As we showed that an exchange of vertical edge
corresponds to a change of coordinates, we are done.
    \end{proof}

\bigskip

\begin{remark}
Even in very good coordinates, the log canonical threshold may
appear at the second (or later) step of the Newton algorithm. Here
is an example:
$$\Ic =((y+x)^5, x^2(y+x)^3,x^8).$$
In these coordinates, the Newton tree is as in Figure 22. One easily
verifies that $lct(\Ic)=\frac 38$, appearing at the second step.
       \begin{figure}[ht]
 \begin{center}
\includegraphics{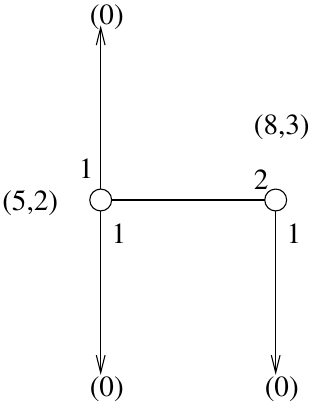}
\caption{}
 \end{center}
 \end{figure}

 \end{remark}

\begin{remark} We should note that the log canonical threshold of the generic curve of an ideal may be different from the log canonical threshold of the ideal. For example, the log canonical threshold of the ideal $\Ic =(x,y)$ is $2$, but the log canonical threshold of the curve $\{x+ay=0\}$ is $1$.
\end{remark}

\bigskip
The next results of \cite{dFEM} and \cite{Co} are immediate
consequences of the previous theorem. Let $\mathcal{M}$ denote the
maximal ideal of $\bc[[x,y]]$. Also, for an ideal $\Ic$ in
$\bc[[x,y]]$ of finite codimension, we denote by $m(\Ic)$ the area
delimited by the axes and the Newton polygon of $\Ic$.

\begin{cor}\label{HSknown}  Assume that $\Ic$ is an ideal in $\bc[[x,y]]$ of finite codimension.
    Let $e(\Ic)$ be its Hilbert-Samuel multiplicity.

(1) Then
    $$e(\Ic)\geq \frac{4}{\text{lct}(\Ic)^2}.$$

(2) We have equality if and only if there is a positive integer $N$
such that the integral closure
 of $\Ic$ is equal to
$\mathcal{M}^N.$
 Moreover, in this case $N=\frac{2}{\text{lct}(\Ic)}$.
 \end{cor}

 \begin{proof}
(1) We have proven that there is an ideal $\Ic'$ isomorphic to $\Ic$
such that the log canonical threshold of $\Ic'$ is the ordinate of
the intersection of the line $\ag=\bg$ with the Newton polygon of
$\Ic'$. Then the Hilbert-Samuel multiplicity of $\Ic'$ is the same
as the Hilbert-Samuel multiplicity of
 $\Ic$. Hence we can assume that $\Ic=\Ic'$.
It is known \cite{CV}  that $e(\Ic)\geq 2m(\Ic)$.  One easily verifies that
$2\frac{1}{\text{lct}(\Ic)^2}\leq m(\Ic)$.

(2) We just derived that equality occurs if and only if
$$e(\Ic)=2m(\Ic)=\frac{4}{\text{lct}(\Ic)^2}.$$
The first equality occurs if and only if $\Ic$ is non-degenerate, the
second one if and only if the Newton polygon of $\Ic$ has exactly
one compact face, with equation of the form $\alpha+\beta=N$. Taking
into account the results of \cite{CV} these statements mean that the
integral closure of $\Ic$ is a power of $\mathcal{M}$.
 \end{proof}

\bigskip
We can give a stronger inequality using the exact computation of the
Hilbert-Samuel multiplicity of an ideal \cite{CV}. There we computed
$e(\Ic)$ using the area of the regions associated with the
successive Newton polygons that appear in the Newton algorithm. If
$\Sg_i=(\sg_1,\cdots, \sg_i)$ is a sequence of Newton maps, then
$\Ic_{\Sg_i}$ is of the form $\Ic_{\Sg_i}=x^{N_i}\Ic'_{\Sg_i}$,
where $\Ic'_{\Sg_i}$ is of finite codimension

\begin{theorem}\label{HSmult}\cite[Theorem 5.18]{CV} Let $\Ic$ be a non-trivial ideal in $\bc[[x,y]]$ of finite codimension  and of depth $d$. Then
$$e(\Ic)=2\, \left(m(\Ic)+\sum_{i=1} ^{d} \sum_{\Sg_i}m(\Ic'_{\Sg_i})\right) ,$$
the second summation being taken over all possible sequences of Newton maps of length $i$.
\end{theorem}

\begin{theorem}\label{HSinequality} Let $\Ic$ be a non-trivial ideal in $\bc[[x,y]]$ of finite codimension. Consider a system of coordinates such that the log canonical threshold is given by its Newton diagram. Then
$$e(\Ic)\geq 4\Big(\frac{1}{\text{lct}(\Ic)^2}+\sum_{\sg}\frac{1}{\text{lct}(\Ic'_{\sg})^2}\Big)$$
where the summation is taken over all Newton maps associated to the faces of the Newton diagram and the roots of the face polynomials.
 \end{theorem}

 \begin{proof}
Theorem \ref{HSmult} can be rewritten as
 $$e(\Ic)=2m(\Ic)+\sum_{\sg}e(\Ic'_{\sg}),$$
where the summation is taken over all Newton maps associated to the
faces of the Newton diagram of $\Ic$ and the roots of the face
polynomials. If moreover we choose a system of coordinates such that
the log canonical threshold is given by the Newton polygon, we have
 $2m(\Ic)\geq 4\frac{1}{\text{lct}(\Ic)^2}$. Using Corollary \ref{HSknown}, we get the result.
 \end{proof}

 As an example, we take Example 3 in \cite{CV}.
We consider in $\bc[[x,y]]$ the ideal
$$\Ic=\left(y^2((x^2+y^3)^2+xy^5)(x^2-y^3),x^8y+x^{12}\right).$$
We computed in \cite{CV} that $e(\Ic)=102$.
There are two possible minimal Newton trees, drawn in Figure 23.

       \begin{figure}[Ht]
 \begin{center}
\includegraphics{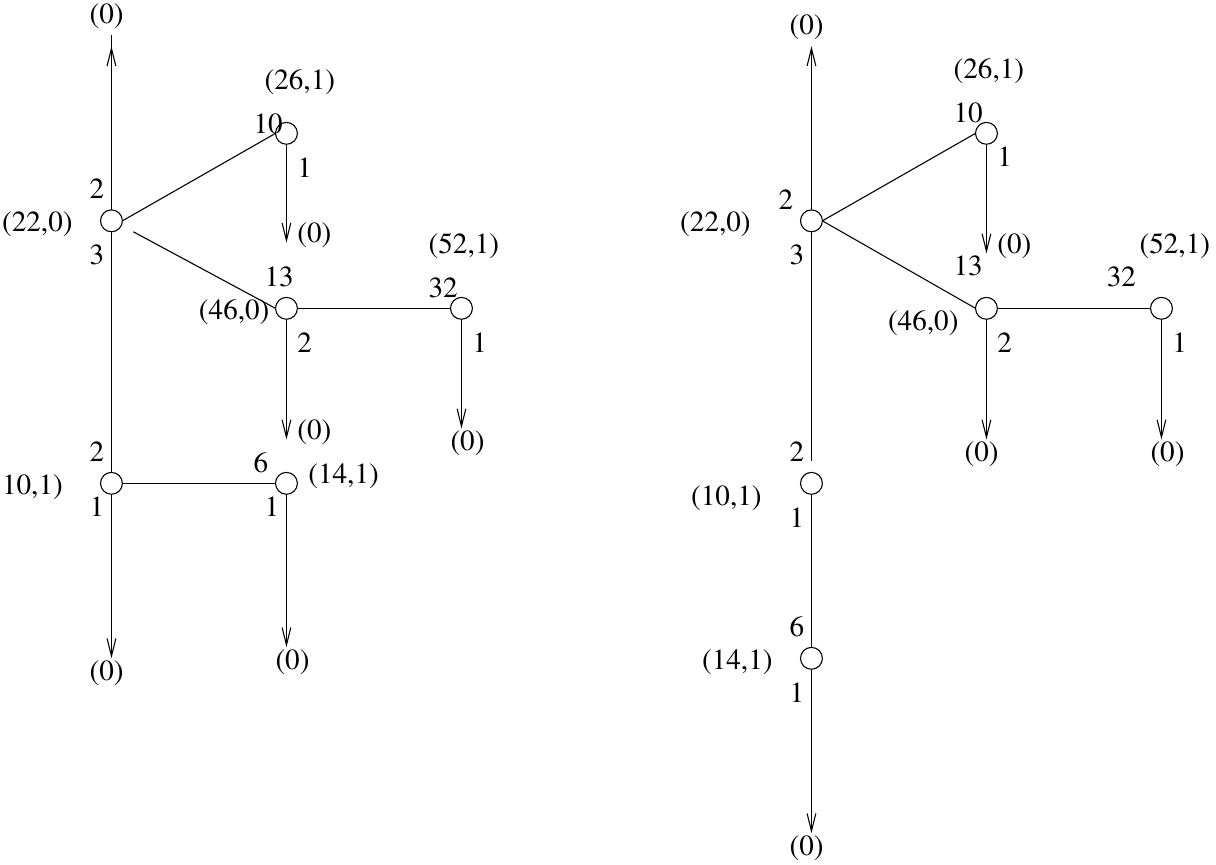}
\caption{}
 \end{center}
 \end{figure}

In both cases the log canonical threshold $\frac 5{22}$ is computed
from the Newton polygon. Corollary \ref{HSknown} applied to this
example gives $102>77,44$. Theorem \ref{HSinequality} gives with the
system of coordinates which gives the Newton tree on the left hand
side $102>85,72$, and if we use coordinates on the right hand side
$102>83.16$.


\bigskip
\begin{thebibliography}{99}

\bibitem{ACLM1} {\sc E. Artal, P. Cassou-Nogu\`es, I. Luengo and A. Melle Hern\'andez} \emph{Quasi-ordinary power series and their zeta functions}, {Mem. Amer. Math. Soc.} {178} (2005), vi+85.

\bibitem{ACLM2} {\sc E. Artal, P. Cassou-Nogu\`es, I. Luengo and A. Melle Hern\'andez} \emph{$\nu$-Quasi-ordinary power series: factorisation, Newton trees and resultants}, {Topology of algebraic verieties and singularities, Contemp. Math.} {538} (2011), 321-343.

\bibitem{ACLM3} {\sc E. Artal, P. Cassou-Nogu\`es, I. Luengo and A. Melle Hern\'andez} \emph{On the log-canonical threshold for germs of plane curves}, {Singularities I, Contemp. Math.} {474} (2008), 1-14.




\bibitem{Co} {\sc A. Corti}
\emph{Singularities of linear systems and $3$-fold birational
geometry}, in \emph{Explicit birational geometry of $3$-folds},
 259-312, Cambridge Univ. Press, Cambridge, 2000.

\bibitem{CLS} {\sc D. Cox, J. Little and H. Schenck}
\emph{Toric varieties},
  Graduate Studies in Mathematics 124, American Mathematical Society, Providence, RI, 2011.

\bibitem{CNL}{\sc P. Cassou-Nogu\`es and A. Libgober} \emph{Multivariate Hodge theorical invariants of germs of plane curves. II}, {The Valuation Theory Home Page, Preprint
228}.

\bibitem{CV} {\sc P. Cassou-Nogu\`es and W. Veys}
\emph{Newton trees for ideals in two variables and applications},
{Proc. London Math. Soc.} {doi:10.1112/plms/pdt047} (2013), 42p.

\bibitem{dFEM} {\sc T. de Fernex, L. Ein and M. Musta\c t\u a}
\emph{Multiplicities and log canonical thresholds}, {J. Alg. Geom.}
{13} (2004), 603-615.

\bibitem{DL1} {\sc J. Denef and F. Loeser}
\emph{Motivic Igusa zeta functions}, {J. Alg. Geom.} {7} (1998),
505-537.




\bibitem{K} {\sc J. Kollar}
\emph{Singularities of pairs}, {Algebraic Geometry-Santa Cruz 1995}, {Proc. Sympos. Pure Math., 62, Part 1}, 221-287


\bibitem{KM} {\sc J. Koll\'ar and S. Mori}
\emph{Birational geometry of algebraic varieties}, {Cambridge Tracts
in Mathematics} {134}, Cambridge University Press, 2008.

\bibitem{KSC} {\sc J. Koll\'ar, K. Smith and A. Corti}
\emph{Rational and nearly rational varieties}, {Cambridge Studies in Advanced Mathematics} {92}, Cambridge University Press, Cambridge, 2004.



\bibitem{LO} {\sc D.T. L\^e and M. Oka}
\emph{On resolution complexity of plane curves}, {Kodai Math. J.}
{18} (1996), 1-36.


\bibitem{Mu} {\sc M. Musta\c t\u a}
\emph{IMPANGA lecture notes on log canonical thresholds},  Notes by Tomasz Szemberg,
 EMS Ser. Congr. Rep., Contributions to algebraic geometry, 407-442, Eur. Math. Soc., Z\"urich, 2012.

\bibitem{NS} {\sc J. Nicaise and J. Sebag}
\emph{Greenberg approximation and the geometry of arc spaces},
{Comm. Alg.} {38} (2010), 4077-4096.




\bibitem{SH} {\sc I. Swanson and C. Huneke}
\emph{Integral closure of ideals, rings, and modules}, {Lond. Math.
Soc. L.N.S.} {336} (2006).


\bibitem{VV1} {\sc L. Van Proeyen and W. Veys}
\emph{Poles of the topological zeta function associated to an ideal
in dimension two}, {Math. Z.} {260} (2008),  615-627.


\bibitem{V1} {\sc  W. Veys}
\emph{Zeta functions for curves and log canonical models}, {Proc.
London Math. Soc.} {74} (1997), 360-378.

\bibitem{V2} {\sc  W. Veys}
\emph{Zeta functions and \lq Kontsevich Invariants\lq\ on singular
varieties}, {Canad. J. Math.} {53} (2001), 834-865.

\bibitem{VZ} {\sc  W. Veys and W. Z\'u\~niga-Galindo}
\emph{Zeta functions for analytic mappings, log-principalization of
ideals, and Newton polyhedra}, {Trans. Amer. Math. Soc.} {360}
(2008), 2205-2227.


\end{thebibliography}
 \end{document}